\documentclass{article}
\usepackage{graphicx} 
\usepackage{amsmath}
\usepackage{amsfonts}
\usepackage{amsthm}
\usepackage{placeins}
\usepackage{caption}
\usepackage{hyperref}
\usepackage{enumitem}
\usepackage{mathabx}
\usepackage{geometry}

\hypersetup{
    colorlinks=true,
    linkcolor=blue,
    filecolor=magenta,      
    urlcolor=cyan,
    pdftitle={Overleaf Example},
    pdfpagemode=FullScreen,
    }

\usepackage{tikz}
\usetikzlibrary{matrix}
\usetikzlibrary{positioning}
\usetikzlibrary {patterns,patterns.meta}
\usetikzlibrary{shapes.geometric}

\pgfdeclarepatternformonly{south east lines}{\pgfqpoint{-0pt}{-0pt}}{\pgfqpoint{8pt}{8pt}}{\pgfqpoint{8pt}{8pt}}{
    \pgfsetlinewidth{0.4pt}
    \pgfpathmoveto{\pgfqpoint{0pt}{8pt}}
    \pgfpathlineto{\pgfqpoint{8pt}{0pt}}
    \pgfpathmoveto{\pgfqpoint{.2pt}{-.2pt}}
    \pgfpathlineto{\pgfqpoint{-.2pt}{.2pt}}
    \pgfpathmoveto{\pgfqpoint{8.2pt}{7.8pt}}
    \pgfpathlineto{\pgfqpoint{7.8pt}{8.2pt}}
    \pgfusepath{stroke}}

\usepackage{titlesec}

\title{Tree and Tripod Nim}
\author{Aidan Hennessey}
\date{September 2023}

\theoremstyle{definition}

\newtheorem{lemma}{Lemma}
\newtheorem{conjecture}{Conjecture}
\newtheorem{theorem}[lemma]{Theorem}
\newtheorem{cor}{Corollary}

\newtheorem{prop}{Proposition}[section]
\newtheorem{remark}{Remark}[section]

\newtheorem{defi}{Definition}[section]
\newtheorem{example}{Example}[section]

\begin{document}

\maketitle

\begin{abstract}
    This paper introduces a variant of the impartial combinatorial game nim, called tree nim, as well as a particular case of tree nim called tripod nim. A certain existence-uniqueness result and a periodicity result are proven about the distribution of $\mathcal{P}$ positions and Grundy values in tree nim. Tripod nim is associated to a family of arrays similar to those which arise in the study of sequential compound games. Using these arrays, a partial analysis is given for tripod nim. Conjectures relating to the periods of the rows of these arrays are put forward.
\end{abstract}

\section*{Introduction}

The field of impartial combinatorial game theory was arguably founded in 1902 with the publishing of the complete analysis of nim.\cite{first-nim} Since then, the field has grown and diversified, yet the game of nim and its variants have remained a central topic of study. One such variant, end nim, is solved in \cite{end-nim-1} with a reformulation of the proof, and further partial Grundy-number results given in \cite{end-nim-2}. According to \cite{unsolved-problems}, Frankel propsed a generalization called Hub-and-Spoke Nim, which Albert notes can be further generalized to a game called forest nim. \cite{end-nim-1} mentions this game and its lack of a solution. Forest nim is a sum of tree nim games. While full analysis remains an open problem, I present some general results about tree nim as well as stronger partial results for tripod nim, the simplest unsolved case of tree nim. In the discussion of tripod nim, a family of arrays arise which are very similar to (but not exactly the same as) the family $\mathcal{A}_*$ studied by Abrams and Cowen-Morton in \cite{alg-arrays}, \cite{colorful-arrays}, and \cite{locator-theorem}.

Section 2 introduces tree nim and builds up to the $\mathcal{P}$-completion lemma, a result which forms the backbone of the rest of the paper. Section 3 introduces the idea of a barrier, which is used to relate different tree nim games. The section concludes with an application of barriers to analyze mis\`ere-nim. Section 4 uses the $\mathcal{P}$-completion lemma to introduce a way one can derive a sequence from a family of tree nim positions. It is then proved that such a sequence is always additively periodic. This is followed by a brief discussion of forest nim and Grundy values. Like games discussed in \cite{fsm-periodicity} and \cite{additive-periodicity}, tree nim's grundy values collect in additively periodic ways. Section 6 introduces tripod nim and explores how the material in sections 3 and 4 specializes to it. Section 7 presents partial results for tripod nim. In section 8, I investigate the periods of the sequences relating to tripod nim, and point out patterns. In section 9, I relate the observed patterns to conjectures about a family of symbolic dynamical systems.

In terms of necessary background, the paper should be largely accessible to a general mathematical audience. The one exception to this is section 6, which discusses Grundy-numbers. Section 6 is not used in anything after, so those unfamiliar with Grundy-numbers can safely ignore it. 

\section{Background}

\subsection{What is nim?}

\par Nim is an impartial combinatorial game played with stacks of coins. Each turn, the player to move player chooses a stack, and then removes some positive number of coins from it (which could possibly be the whole stack). The game ends when the final coin is removed, and the player who makes that final move wins. Below is a example of how a game of nim might go:

\begin{figure}[!h]
    \centering
    \includegraphics[scale=0.42]{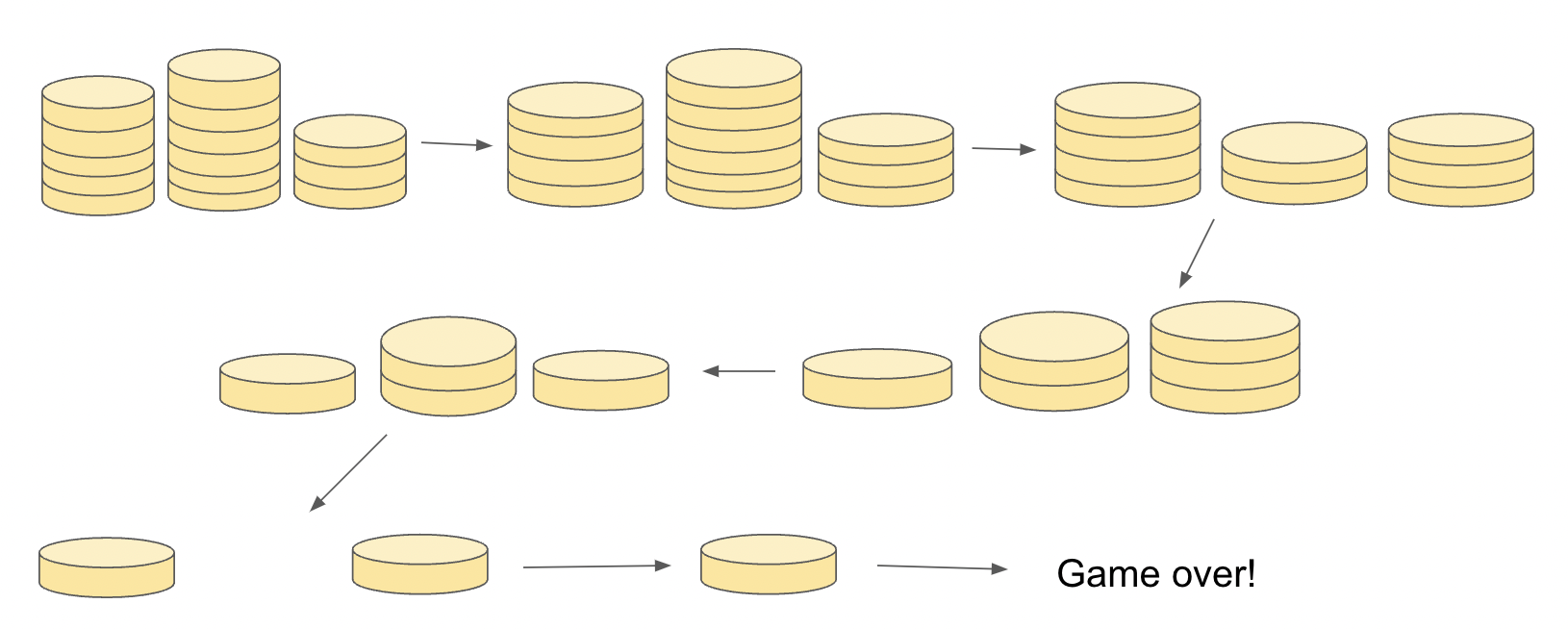}
    \caption{An example nim game}
    \label{nim_game}
\end{figure}

\par Like any impartial combinatorial game, nim's positions can all be classified as one of the following:

\begin{itemize}
    \item a $\mathcal{P}$ position, where the player who $\mathcal{P}$reviously moved can win by force
    \item an $\mathcal{N}$ position, where the player who moves $\mathcal{N}$ext has a winning strategy
\end{itemize}

The classification of an impartial combinatorial game position as $\mathcal{P}$ or $\mathcal{N}$ can be computed by applying the following recursive criteria:

\begin{itemize}
    \item If there is a legal move from position $A$ to some $\mathcal{P}$-position, $A$ is an $\mathcal{N}$-position. Otherwise it is a $\mathcal{P}$-position.
    \item In particular, a position from which there are no legal moves is $\mathcal{P}$.
\end{itemize}

In general, a winning strategy in an impartial combinatorial game is to always move to $\mathcal{P}$ positions. Your opponent must then move to an $\mathcal{N}$ position, allowing you a $\mathcal{P}$ response. In nim, the $\mathcal{P}$ positions follow a very nice pattern.

\subsection{How do you win?}

\par Nim has a complete analysis, first published in \cite{first-nim}. A nim position is $\mathcal{P}$ if and only if the \emph{nim-sum} of its stack sizes is 0. The nim-sum (denoted $\oplus$), also called 2-adic sum, of two natural numbers can be found by writing both numbers in binary and taking their bitwise XOR, or equivalently by adding in base 2 without carrying.
\par For example, $11 \oplus 13 = 1011 \oplus 1101 = 0110 = 6$. It's easy to check that this operation is associative and commutative, and that $a \oplus b = 0 \iff a=b$. Table 1 displays the nim-sums of the numbers 0 to 15.

\begin{table}[h]
    \centering
    \begin{tabular}{c|c c c c c c c c c c c c c c c c}
        $\oplus$ &0&1&2&3&4&5&6&7&8&9&10&11&12&13&14&15\\
        \hline
        0&0&1&2&3&4&5&6&7&8&9&10&11&12&13&14&15\\
        1&1&0&3&2&5&4&7&6&9&8&11&10&13&12&15&14\\
        2&2&3&0&1&6&7&4&5&10&11&8&9&14&15&12&13\\
        3&3&2&1&0&7&6&5&4&11&10&9&8&15&14&13&12\\
        4&4&5&6&7&0&1&2&3&12&13&14&15&8&9&10&11\\
        5&5&4&7&6&1&0&3&2&13&12&15&14&9&8&11&10\\
        6&6&7&4&5&2&3&0&1&14&15&12&13&10&11&8&9\\
        7&7&6&5&4&3&2&1&0&15&14&13&12&11&10&9&8\\
        8&8&9&10&11&12&13&14&15&0&1&2&3&4&5&6&7\\
        9&9&8&11&10&13&12&15&14&1&0&3&2&5&4&7&6\\
        10&10&11&8&9&14&15&12&13&2&3&0&1&6&7&4&5\\
        11&11&10&9&8&15&14&13&12&3&2&1&0&7&6&5&4\\
        12&12&13&14&15&8&9&10&11&4&5&6&7&0&1&2&3\\
        13&13&12&15&14&9&8&11&10&5&4&7&6&1&0&3&2\\
        14&14&15&12&13&10&11&8&9&6&7&4&5&2&3&0&1\\
        15&15&14&13&12&11&10&9&8&7&6&5&4&3&2&1&0\\
    \end{tabular}
    \caption{Nim-Addition table}
    \label{tab:nim_addition}
\end{table}

\par To prove this analysis, it suffices to show that (1) from every position with nim-sum 0, any subsequent move has non-0 nim-sum and (2) from any position with non-zero nim-sum, there is a move to a position with nim-sum 0.

\begin{prop}
    There are no moves between positions with nim-sum 0. 
\end{prop}

\begin{proof}

All moves in nim change the size of exactly one stack. Consider two positions which differ in the size of at most one stack, so positions with stack sizes $n_1, n_2, n_3, ... n_k$ and $n_1', n_2, n_3, ..., n_k$, such that 

\begin{equation*}
    n_1 \oplus n_2 \oplus ... \oplus n_k = n_1' \oplus n_2 \oplus ... \oplus n_k = 0.
\end{equation*}

Now consider 
\begin{equation*}
    \begin{split}
         0 & = 0 \oplus 0 \\
         & = (n_1 \oplus n_2 \oplus ... \oplus n_k) \oplus (n_1' \oplus n_2 \oplus ... \oplus n_k) \\
         & = (n_1 \oplus n_1') \oplus (n_2 \oplus n_2) \oplus ... \oplus (n_k \oplus n_k) \\
         & = (n_1 \oplus n_1') \oplus 0 \oplus 0 \oplus ... \oplus 0 \\
         & = n_1 \oplus n_1'
    \end{split}
\end{equation*}

As noted earlier, we may conclude that $n_1=n_1'$. Therefore, no two games with nim-sum 0 may differ in exactly one stack, and so there are no moves between positions with nim-sum 0.
    
\end{proof}

\begin{defi}[to see]
    Position $A$ \textit{sees} position $B$ if it is legal to move from $A$ to $B$. Denote this relation $A \rightarrow B$.
\end{defi}

\begin{prop}
    Every position with non-zero nim-sum sees a 0 nim-sum position.
\end{prop}

\begin{proof}
    Suppose the nim-sum $s$ of a position with stacks $n_1, n_2, ..., n_k$ has $d$ digits when written in binary. Then, it has a 1 in the $2^{d-1}$'s place. Thus, if we write all the stack sizes in binary, there is an odd (and hence, non-zero) number of stacks which have a 1 in the $2^{d-1}$'s place. Pick such a stack, say $n_i$, and consider $n_i \oplus s$. This number agrees with $n_i$ in all its binary digits to the left of the $2^{d-1}$'s place, but then has a 0, where $n_i$ has a 1. We can hence conclude $n_i \oplus s < n_i$. This means we can move from $n_i$ to $n_i \oplus s$. The following calculation shows that such a move results in a nim-sum 0 position, completing the proof.

    \begin{equation*}
        n_1 \oplus ... \oplus (n_i \oplus s) \oplus ... n_k = (n_1 \oplus ... \oplus n_k) \oplus s = s \oplus s = 0
    \end{equation*}
\end{proof}

\subsection{Mis\`ere Nim}

When a combinatorial game is defined such that the player who makes the last move wins, it is said the game is in normal play. If this rule is flipped so that the player who makes the last move \textit{loses}, then the game is called a \textit{mis\`ere game}. Mis\`ere nim is, as the name suggests, the mis\`ere form of nim. The allowed moves are the same, but in mis\`ere nim the player who takes the last coin loses. When every stack has size 1, there are no meaningful choices to be made, and so every position is either $\mathcal{P}$ in normal play and $\mathcal{N}$ in mis\`ere play or vice-versa. Amazingly, however, as long as there is one stack with multiple coins, a mis\`ere nim position is $\mathcal{P}$ if and only if it is $\mathcal{P}$ under normal play. This result is far from new, but section 10 provides a reformulation of the proof using machinery developed in this paper.

\section{Tree Nim and the $\mathcal{P}$-completion Lemma}

Suppose that instead of treating every stack equally, we instead construct a tree, with a stack at each vertex. Players may only take from stacks at leaves, with inner stacks becoming playable once outer ones are exhausted. A tree nim position can be understood to be a tree with natural number labels on the vertices, referred to as the sizes $|v_i|$ of the vertices $v_i$. 

\FloatBarrier
\begin{figure}[!h]
    \centering

    \begin{tikzpicture}
        \node[rectangle, color=red, text=black, thick, draw] (1) at (0,0) {1};
        \node[above=of 1, rectangle, color=red, text=black, thick, draw] (red4) {4};
        \node[rectangle, color=red, text=black, thick, draw] (red8) at (-1.5, 0.4) {8};
        \node[above=of red8, circle, color=green, text=black, thick, draw] (green4) {4};
        \node[circle, color=green, text=black, thick, draw] (green8) at (1.5, 0.4) {8};
        \node[above=of green8, circle, color=green, text=black, thick, draw] (6) {6};
        \node[circle, color=green, text=black, thick, draw] (3) at (1, -1) {3};
        \node[left=of 3, circle, color=green, text=black, thick, draw] (2) {2};
        \node[circle, color=green, text=black, thick, draw] (5) at (-2, -0.8) {5};

        \draw (green4)--(red4);
        \draw (red4)--(6);
        \draw (red4)--(1);
        \draw (green8)--(1);
        \draw (red8)--(1);
        \draw (red8)--(5);
        \draw (1)--(2);
        \draw (1)--(3);
    \end{tikzpicture}
    
    \caption{An example tree nim position. Vertices circled in green are in play. Vertices boxed in red are not yet in play.}
    \label{tree_nim}
\end{figure}
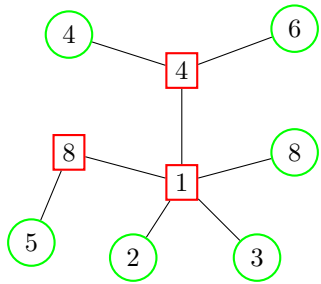
\FloatBarrier

\begin{defi}[Ray]
    Pick a tree nim position $r_0$ and a vertex $v \in V(r_0)$. A ray is the set $r_0\#_v l = \{r_n: n \in \mathbb{N}_0\}$,\footnote{Throughout this paper, $\mathbb{N}$ refers to the strictly positive integers, while $\mathbb{N}_0$ refers to the non-negative integers.} where $r_n$ for $n > 0$ is formed by adjoining a leaf $l$ of size $n$ to $r_0$ at $v$. $l$ is called the \textit{variable leaf}.
\end{defi}

\begin{example}
    Start with the example tree nim position in Figure 2. Remove the leaf with size 3. Let this be $r_0$. Call the vertex with size 1 $v$. The position $r_n$ (for $n>0$) belonging to the ray $r_0\#_v l$ would then look like the following:
    
    \centering
    \begin{tikzpicture}
        \node (1) at (0,0) {1};
        \node[above=of 1] (red4) {4};
        \node (red8) at (-1.5, 0.4) {8};
        \node[above=of red8] (green4) {4};
        \node (green8) at (1.5, 0.4) {8};
        \node[above=of green8] (6) {6};
        \node (n) at (1, -1) {$n$};
        \node[left=of n] (2) {2};
        \node (5) at (-2, -0.8) {5};

        \draw (green4)--(red4);
        \draw (red4)--(6);
        \draw (red4)--(1);
        \draw (green8)--(1);
        \draw (red8)--(1);
        \draw (red8)--(5);
        \draw (1)--(2);
        \draw (1)--(n);
    \end{tikzpicture}
\end{example}

In particular, the position in figure 2 appears in the ray as $r_3$.

\begin{defi}[to see, as a ray]
    Let $R=\{r_n\}$ and $R'=\{r'_n\}$ be rays. $R$ sees $R'$ (denoted $R \rightarrow R'$) if $\forall n \in \mathbb{N}_0$, $r_n \rightarrow r'_n$.
\end{defi}

\begin{remark}
Every ray $R=\{r_n: n \in \mathbb{N}_0\}$ satisfies $r_n \rightarrow r_m \iff n > m$. Also, $r_0$ has one fewer vertices than $r_n$ for all $n>0$. There is an analogous relationship between positions and rays in tree nim and points and lines in euclidean space. This is formalized in propositions 4.1 and 4.2.
\end{remark}

\begin{prop}
    For any tree nim positions $p$ and $q$ where $p \rightarrow q$, there is a unique ray containing $p$ and $q$.
\end{prop}

\begin{proof}
Suppose $R=\{r_n\}$ is a ray containing $p$ and $q$. 

If the move from $p$ to $q$ is the complete removal of a leaf $l$ based at $v$, then $q$ must be $r_0$, by remark 4.1. Additionally, the variable leaf must be the one removed during the move from $p$ to $q$, or else $r_n$ would have a different tree structure than $p$ for all $n$, and hence be a different tree nim position. $R=r_0 \#_v l$ is thus the only possibility, giving usiqueness. It works, as letting $s = |l|$, we see $p=r_s$. Thus in this case we also have existence.

If the move from $p$ to $q$ is not the complete removal of a leaf but rather a decrease in the size of some leaf $l$ from $n$ to $m$, then neither $p$ nor $q$ is $r_0$. Because all $p_k$ are the same except for the variable leaf, the variable leaf must be $l$. Letting $v$ be $l$'s neighbor vertex,  it follows that $(p \backslash l) \#_v l$ is the only possible satisfactory $R$, giving uniqueness. Indeed, it is a ray containing $p$ and $q$, as $p=r_n$ and $q=r_m$, so we have existence as well.
\end{proof}

\begin{prop}
    Given a ray $R=r_0 \#_v l$ and tree nim position $p \notin R$ such that $r_n \rightarrow p$ for some $n>0$, there exists a unique ray $R'$ such that $R \rightarrow R'$ and $p \in R'$. Additionally, $p=r'_n$.
\end{prop}

\begin{proof}

I first show that for any satisfactory $R' = \{r'_n\}_{n=0}^\infty$, that $p=r'_n$. To see this, suppose otherwise, i.e. that $p=r'_k$ for $k \neq n$. Then, we would have $r_k \rightarrow p$. Let $w$ be the vertex at which $r_n$ and $p$ differ. $r_k$ differs from $p$ at both $l$ and $w$, so $r_k$ cannot possibly see $p$, which yields a contradiction.

Let $q$ be the position which is obtained from $r_{n+1}$ by decreasing the size of $w$ to match $|w|_p$. $r_{n+1} \rightarrow q \rightarrow p$. By prop 4.1, there is a unique ray $R'$ containing $p$ and $q$. It is easy to see that $R \rightarrow R'$, giving existence. 

For any ray $R'' \leftarrow R$ with $r''_n=p$, we would need to have $r_{n+1} \rightarrow r''_{n+1} \rightarrow p$. $r_n$ and $q$ are the only positions satisfying this condition on $r''_n$, and $r_n$ clearly won't work, so $r''_n=q$. $R''$ then contains both $p$ and $q$. Thus, by prop 2.1, $R''=R'$, giving uniqueness.

    
\end{proof}

\begin{defi}[Leaf Sum]
    The \textit{leaf sum} $|R|$ of a ray $R = r_0 \#_v l_0$ is the sum $\sum_{i\neq 0} |l_i|$ over leaves $l_i$ of $r_1$, excluding $l_0$. Because $l_0$ is excluded from the calculation, any $r_n$ for $n>0$ could replace $r_1$ and yield an equivalent definition.
\end{defi}

\begin{example}
    $R$ and $R'$ are tree nim rays with $R \rightarrow R'$, and $r_n$ and $r'_n$ are positions in the respective rays. Below are $r_n$, on the left, and $r'_n$, on the right. $|R|=25$ and $|R'|=22$.

    \centering
    \begin{tikzpicture}
        \node (1) at (0,0) {1};
        \node[above=of 1] (red4) {4};
        \node (red8) at (-1.5, 0.4) {8};
        \node[above=of red8] (green4) {4};
        \node (green8) at (1.5, 0.4) {8};
        \node[above=of green8] (6) {6};
        \node (n) at (1, -1) {$n$};
        \node[left=of n] (2) {2};
        \node (5) at (-2, -0.8) {5};

        \draw (green4)--(red4);
        \draw (red4)--(6);
        \draw (red4)--(1);
        \draw (green8)--(1);
        \draw (red8)--(1);
        \draw (red8)--(5);
        \draw (1)--(2);
        \draw (1)--(n);

        \node (1_) at (6,0) {1};
        \node[above=of 1_] (red4_) {4};
        \node (red8_) at (4.5, 0.4) {8};
        \node[above=of red8_] (green4_) {4};
        \node (green8_) at (7.5, 0.4) {5};
        \node[above=of green8_] (6_) {6};
        \node (n_) at (7, -1) {$n$};
        \node[left=of n_] (2_) {2};
        \node (5_) at (4, -0.8) {5};

        \draw (green4_)--(red4_);
        \draw (red4_)--(6_);
        \draw (red4_)--(1_);
        \draw (green8_)--(1_);
        \draw (red8_)--(1_);
        \draw (red8_)--(5_);
        \draw (1_)--(2_);
        \draw (1_)--(n_);
    \end{tikzpicture}
\end{example}

\begin{prop}
    Any ray $R$ sees exactly $|R|$ rays.
\end{prop}

\begin{proof}
Let $R = \{r_n\}$. Each $R'$ such that $R \rightarrow R'$ has a distinct $r'_1$ such that $r_1 \rightarrow r'_1$. Additionally, any $p$ other than $r_0$ satisfying $r_q \rightarrow p$ is the $r'_1$ for some $R'$ seen by $R$, by the proposition 4.2. Thus, the number of seen rays is equal to the number of available moves from $r_1$, excluding the move to $r_0$. Each move reduces the size of some non-variable leaf $l_i$. For every leaf $l_i$, there are $|l_i|$ non-negative integer sizes it could be reduced to. Summing over all non-variable leaves gives $|R|$, by definition.
\end{proof}

\begin{defi}[$\mathcal{P}$-completion]
    A $\mathcal{P}$-completion of a ray $R=\{r_n\}$ is a value $k$ such that $r_k$ is a $\mathcal{P}$ position.
\end{defi}

\begin{lemma}[$\mathcal{P}$-Completion Lemma]
    Every ray $R = \{r_n\}$ has a unique $\mathcal{P}$-completion $\mathcal{P}(R)$. Furthermore, $\mathcal{P}(R) \leq |R| + 1$.
\end{lemma}

\begin{proof}
Uniqueness: By way of contradiction, suppose a ray $R = \{r_k\}$ has two distinct completions $n$ and $m$, with $n>m$. Then, $p_n$ and $p_m$ would both be $\mathcal{P}$. $p_n \rightarrow p_m$, contradicting the fact that $\mathcal{P}$ positions cannot see $\mathcal{P}$ positions.

Existence: By way of contradiction, suppose $r_n$ is $\mathcal{N}$ for all $n \leq |R| + 1$. Then, $r_0, r_1, ..., r_{|R|+1}$ see corresponding $\mathcal{P}$ positions $p_0, p_1, ..., p_{|R|+1}$. Because no $r_n$ is $\mathcal{P}$, each of $p_0, p_1, ..., p_{|R|+1}$ are outside $R$. Thus, by proposition 4.2, each of $p_1, p_2, ..., p_{|R|+1}$ live in some ray $R_n$ seen by $R$. There are $|R|+1$ relevant $p_n$'s, but $R$ only sees $|R|$ rays by prop 4.3. Thus, by the pigeon hole principle, there exists an $R_n$ which contains at least two $\mathcal{P}$ positions. This contradicts the uniqueness of $\mathcal{P}$-completions.
\end{proof}

\section{Barriers and Shadows}

In this chapter I introduce a tool for showing equivalence between different tree nim games.

\begin{defi}[A tree's world]
    Fix a tree $T$, and fix sizes for its non-leaf nodes. Define the \textit{world} $W(T)$ of $T$ to be the set of tree nim positions with this fixed structure, along with the positions they see.
\end{defi}

\begin{defi}[Interior, exterior positions]
    Given a world $W$ for a selection of inner sizes of a tree $T$, positions in $W$ which have $T$ as their tree structure are called \textit{interior}. Other positions in $W$ are called \textit{exterior}. Exterior positions, as they are seen by interior positions, have tree structures that can be obtained by deleting a leaf of $T$. Denote the set of interior positions by $W^\circ$ and the exterior by $\partial W$.
\end{defi}

\begin{defi}[A world's dimension]
    The \textit{dimension} $d(W)$ of a world $W$ is the number of leaves on its corresponding tree.
\end{defi}

Each interior position in a world can be specified by the sizes of its leaves, and every set of leaf sizes yields an interior position. If you allow your set of leaf sizes to have up to one 0, then this correspondence extends to the entire world. Order the leaves of $T$ $l_1, l_2, ..., l_d$. Then, for a $d$ dimensional world $W$, we have the following natural bijection:
\[ l_W: W \longrightarrow \{ (n_1, n_2, ..., n_d): n_i \in \mathbb{N}_0 \text{ and } n_i=0 \text{ for at most one } i \} \]
given by sending each position to the ordered tuple of its leaf sizes. This bijection lifts to a bijection $l_W^*$ between rays in $W$ and rays in $\mathbb{N}_0^d$.\footnote{A ray in $\mathbb{N}_0^d$ is understood to be a set $\{(a_0, a_1, ..., a_{k-1}, n, a_{k+1}, ..., a_d) : n \geq 0\}$ for any selection of strictly positive $a_i$'s.}

\begin{defi}[Lattice map, Lattice Representation]
    The function $l_W$, defined above, is called the \textit{lattice map}. The \textit{lattice representation} $L(W)$ of $W$ is the image of the $\mathcal{P}$ positions under $l_W$. It is a subset of $\mathbb{N}_0^d$.
\end{defi}


\begin{defi}[Barrier]
A \textit{barrier} $B$ is a partition of a world $W$ into a disjoint union $U_B \sqcup L_B = W$ such that $U_B$ (referred to as the upper set, or set above the barrier) satisfies the following two properties:
\begin{itemize}[topsep=0pt, itemsep=-1ex]
    \item $U_B$ contains all exterior positions 
    \item If a position $p \in U_B$ sees some position $q$, then $q \in U_B$.
\end{itemize}
These are called the barrier axioms. It's worth noting that given two subsets $U$ and $U'$ satisfying the barrier axioms, $U \cup U'$ and $U \cap U'$ satisfy the axioms as well.
\end{defi}

\begin{defi}[Shadow]
    Let $B$ be a barrier in $W$. For each ray $R = \{ r_n \} \subset W$, if there is any $n$ such that $r_n \in L_B$, there is a least such $n$. If $\mathcal{P}(R) < n$, then we say $R$ is \textit{shadowed}. The set of shadowed rays is called the \textit{shadow} $S(B)$ of the barrier.
\end{defi}

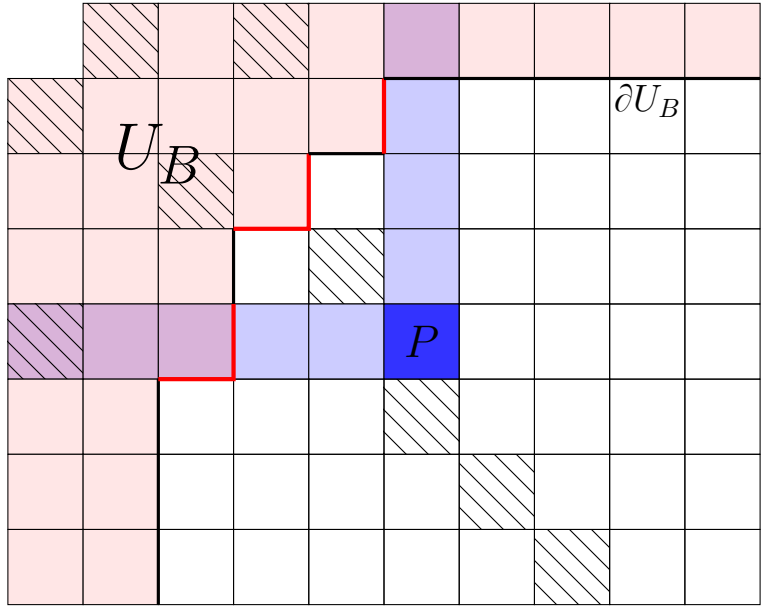
\begin{figure}[!htb]
    \centering
    \begin{tikzpicture}[
    upperp/.style={rectangle, preaction={fill, red!10}, minimum width=1cm, minimum height=1cm, pattern=south east lines, pattern color = black, draw},
    uppern/.style={rectangle, fill=red!10, minimum width=1cm, minimum height=1cm, draw},
    lowerp/.style={rectangle, minimum width=1cm, minimum height=1cm, pattern=south east lines, pattern color = black, draw},
    seenupperp/.style={rectangle, preaction={fill, violet!30}, minimum width=1cm, minimum height=1cm, pattern=south east lines, pattern color = black, draw},
    seenuppern/.style={rectangle, fill=violet!30, minimum width=1cm, minimum height=1cm, draw},
    seenlowern/.style={rectangle, fill=blue!20, minimum width=1cm, minimum height=1cm, draw},
    seer/.style={rectangle, fill=blue!80, minimum width=1cm, minimum height=1cm},
    empty/.style={rectangle, minimum height=1cm, minimum width=1cm, draw=none}
    ]

        \matrix at (0, 0) [matrix of nodes, 
                           nodes in empty cells,
                           nodes={rectangle, minimum width=1cm, minimum height=1cm, draw}, row sep=-\pgflinewidth, column sep=-\pgflinewidth] 
        {
            |[empty]| & |[upperp]| & |[uppern]| & |[upperp]| & |[uppern]| & |[seenuppern]| & |[uppern]| & |[uppern]| & |[uppern]| & |[uppern]| \\
            |[upperp]|& |[uppern]| & |[uppern]| & |[uppern]| & |[uppern]| & |[seenlowern]| & & & & \\
            |[uppern]| & |[uppern]| & |[upperp]| & |[uppern]| & & |[seenlowern]| & & & & \\
            |[uppern]| & |[uppern]| & |[uppern]| & & |[lowerp]| & |[seenlowern]| & & & & \\
            |[seenupperp]| & |[seenuppern]| & |[seenuppern]| & |[seenlowern]| & |[seenlowern]| & |[seer]| & & & & \\
            |[uppern]| & |[uppern]| & & & & |[lowerp]| & & & & \\
            |[uppern]| & |[uppern]| & & & & & |[lowerp]| & & & \\
            |[uppern]| & |[uppern]| & & & & & & |[lowerp]| & & \\
        };

        \node at (-3, 2) {\Huge $U_B$};
        \node at (0.5, -0.5) {\LARGE $P$};
        \node at (3.5, 2.7) {\Large $\partial U_B$};

        \draw[very thick] (-3, -4)--(-3, -1);
        \draw[ultra thick, color=red] (-3, -1)--(-2, -1)--(-2, 0);
        \draw[very thick] (-2, -0)--(-2, 1);
        \draw[ultra thick, color=red] (-2, 1)--(-1, 1)--(-1, 2);
        \draw[very thick] (-1, 2)--(0, 2);
        \draw[ultra thick, color=red] (0, 2)--(0, 3);
        \draw[very thick] (0, 3)--(5, 3);
        
    \end{tikzpicture}
    \caption{This diagram shows an example barrier in a 2D game. The game's $\mathcal{P}$ positions are hatched. The region $U_B$ is tinted red. There is a correspondence between rays which intersect the lower set, and segments along the boundary between the upper and lower sets. One can imagine light sources to the north and west shining down on the boundary, with the $\mathcal{P}$ positions casting shadows. The shadowed segments are colored red. An arbitrary position $P$ is colored blue, and the positions it sees are tinted blue. $P$ is an $\mathcal{N}$ position because it sees the $\mathcal{P}$ position to its left. Alternatively, $P$ is an $\mathcal{N}$ position because it is in a shadowed ray, and below the barrier.}
    \label{fig:enter-label}
\end{figure}

\FloatBarrier

\begin{defi}[isomorphism of barriers]
    Given two $d$-dimensional worlds $W$ and $V$, and barriers $B$ of $W$ and $C$ of $V$, let $(W, B)$ and $(V, C)$ be isometric if two conditions hold:
    \begin{enumerate}
        \item There exists some translation $t: \mathbb{Z}^d \rightarrow \mathbb{Z}^d$ such that $l_W(L_B) = t(l_V(L_C))$
        \item For all rays $R \subset W$ and $R' \subset V$, if $l_W(R) = t(l_V(R'))$, then $R \in S(B) \iff R' \in S(C)$.
    \end{enumerate}
    Informally, isometric barriers are the "same" partition upto some shift, and have the "same" shadow after applying that shift, from the perspective of the lattice.
\end{defi}

\begin{example}

Somewhat abusively, $n$-stack nim can be thought of as tree nim, where there are $n$ leaves protruding from a central vertex of size 0. Similarly, $n$-stack mis\`ere nim can be understood as tree nim, where there are $n$ leaves protruding from a central vertex of size 1. These each have corresponding $n$-dimensional worlds, and there is an isometry of barriers between them. The barriers $B$ and $C$ each contain the exterior, and the position in which every leaf has size 1. Figure 4 illustrates the barriers for $n=2$.

\begin{figure}[!h]
    \centering
    \begin{tikzpicture}[p/.style={rectangle, minimum width=1cm, minimum height=1cm, pattern=south east lines, pattern color = black, draw}, e/.style={rectangle, minimum width=1cm, minimum height=1cm, draw=none}]

        \matrix (normal) at (0, 0.5) [matrix of nodes, nodes in empty cells, nodes={rectangle, minimum width=1cm, minimum height=1cm, draw}, row sep=-\pgflinewidth, column sep=-\pgflinewidth] {
                |[e]| & & & & \\
                & |[p]| & & & \\
                & & |[p]| & & \\
                & & & |[p]| & \\
                & & & & |[p]| \\
                };
        
        \matrix (misere) at (7, 0.5) [matrix of nodes, nodes in empty cells, nodes={rectangle, minimum width=1cm, minimum height=1cm, draw}, row sep=-\pgflinewidth, column sep=-\pgflinewidth] {
            |[e]| & |[p]| & & & \\
            |[p]| & & & & \\
            & & |[p]| & & \\
            & & & |[p]| & \\
            & & & & |[p]| \\
        };
    
        \draw[very thick] (-1.5, -2)--(-1.5, 1);
        \draw[very thick] (5.5, -2)--(5.5, 1);
        \draw[ultra thick, color=red] (-1.5, 1)--(-0.5, 1)--(-0.5, 2);
        \draw[ultra thick, color=red] (5.5, 1)--(6.5, 1)--(6.5, 2);
        \draw[very thick] (-0.5, 2)--(2.5, 2);
        \draw[very thick] (6.5, 2)--(9.5, 2);
    
        \node at (2.9, 2.1) {\LARGE$B$};
        \node at (9.9, 2.1) {\LARGE$C$};
        
    \end{tikzpicture}
    \caption{On the left we have a plot of the $\mathcal{P}$ positions of 2 stack normal nim, and on the right is a plot of the $\mathcal{P}$ positions of 2 stack mis\`ere nim. The same shadow (in red) is cast onto each barrier (represented by the bold black line). Because the cast shadows are the same, the plots will look identical below the barrier.}
    \label{fig:enter-label}
\end{figure}
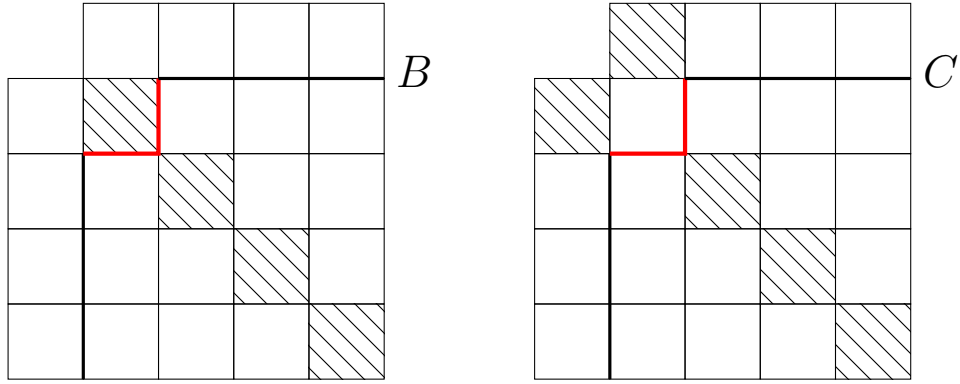

To see that these barriers are indeed isomorphic for all $n$, i.e. that the same shadow is cast on each, requires an inductive argument which makes use of the following theorem.

\end{example}

\begin{theorem}[Barrier Isomorphism Theorem]
    Let $W$ and $V$ be worlds with barriers $B$ and $C$ such that there is an isometry of barriers between $(W, B)$ and $(V, C)$. Let $t$ be the translation such that $l_W(L_B)=t(l_V(L_C))$. Then, if $p\in L_B$ is a position below the barrier, it is $\mathcal{P}$ if and only if $l_V^{-1}tl_W(p)$ is $\mathcal{P}.$
\end{theorem}

\begin{proof}
Let $p \in W$ be a position in the lower set $L_B$ which only sees positions in the upper set $U_B$. Call such a position a corner position. I claim $p$ is $\mathcal{N}$ if and only if it is contained in a shadowed ray. If any ray containing $p$ is shadowed, then the ray has a $\mathcal{P}$ position in the upper set, which must be seen by $p$. This would make $p$ an $\mathcal{N}$ position. Conversely, by prop 4.1, for each position $q$ satisfying $p\rightarrow q$ there is a ray $R_q$ containing $p$ and $q$. If $p$ is $\mathcal{N}$, there is a $\mathcal{P}$ position $q$ earlier than $p$ in some ray $R_q$. Because every position earlier than $\mathcal{P}$ is upper, $R_q$ is shadowed, and thus $p$ lies in a shadowed ray.

If $p$ lies in a shadowed ray $R$, the corresponding ray $R'=l_V^{*-1}tl^*_W(R)$ is also shadowed, by the second barrier isomorphism condition. $l_V^{-1}tl_W(p) \in R'$, so $l_V^{-1}tl_W(p)$ also lies in a shadowed ray. Repeating the same argument in the other direction yields an if and only if. Thus, the isomorphism theorem holds for corner positions. Towards a proof for non-corner positions, construct new barriers $B'$ and $C'$ by taking the unions of $B$ and $C$ with their corner positions. I claim these new barriers are also isomorphic with $t$ as the shift. This has two parts:
\begin{enumerate}
    \item $l_W(L_{B'}) = tl_V(L_{C'})$: Let $p \in L_C$ be a lower position and let $l_V(p) = (a_1, ..., a_d)$. $p$ is a corner position if and only if each of $l_V^{-1}(l_V(p)-e_i)$, where $e_i$ is a basis vector $(0, ..., 0, 1, 0, ..., 0)$, is upper.\footnote{It is well defined to subtract a basis vector from a lower position, because the lower set lies in the interior, and thus all the $a_i$ are strictly positive. The barrier axioms require that $U_B$ contains the exterior in part to ensure that this step is valid.} Because $B$ and $C$ are isometric barriers, this happens if and only if each $l_W^{-1}(tl_V(p)-e_i)$ is upper for all $i$, and hence if and only if $l^{-1}tl(p)$ is a corner position for $L_B$. Thus, $t$ takes corner positions to corner positions, and symmetrically, so does $t^{-1}$. It follows that $l_W(L_{B'})=tl_V(L_{C'})$.
    
    \item $R \in S(B) \iff l^{*-1}_Vtl^*_W(R) \in S(C)$: Let $R' = l^{*-1}_Vtl^*_W(R)$.  If $R \cap L_B = R \cap L_{B'}$, then $R \in S(B)$ if and only if $R \in S(B')$. Additionally $R' \cap L_C = R' \cap L_{C'}$, so $R' \in S(C)$ if and only if $R' \in S(C')$ By the isomorphism between $B$ and $C$, we then have $$R \in S(B') \iff R \in S(B) \iff R' \in S(C) \iff R' \in S(C').$$

    If $R \cap L_B \neq R \cap L_{B'}$, then there are two cases: \newline
    (i) $R \in S(B)$. Then, by the isomorphism between $B$ and $C$, $R' \in S(C)$. These imply $R \in S(B')$ and $R' \in S(C')$ respectively. \newline
    (ii) $R \not\in S(B)$. In this case, $R \cap L_{B'} \backslash L_B)$ contains one position $p$, and it is a corner position. $R' \cap (L_{C'}\backslash L_C)$ contains only the corresponding corner position $p' = l_V^{-1}tl_W(p)$. $R \in S(B')$ if and only if $p$ is $\mathcal{P}$. Similarly, $R' \in S(C')$ if and only if $p'$ is $\mathcal{P}$. Using the isomorphism theorem for corner positions, we get $$R \in S(B') \iff p \text{ is } \mathcal{P} \iff p' \text{ is } \mathcal{P} \iff R' \in S(C')$$ as desired.
\end{enumerate}
$W$ and $V$ thus have new isomorphic barriers $B'$ and $C'$. The theorem holds for the corner positions of $B'$ and $C'$, by the same argument as before. We can create new isomorphic barriers $B''$ and $C''$ by adding in the corner positions of $B'$ and $C'$, and iterate this process as long as we please. I claim each lower position appears as a corner in a finite number of iterations. Once this is established, the theorem will be proven.

Let $p \in L_B$ be a lower position in $W$. Let $l_W(p)=(a_1, a_2, ..., a_d)$, and consider the region $$l_W(L_B) \cap \bigtimes_{i=1}^d [1, a_i]$$ in the lattice representation. It contains a finite number of cells, bounded above by the product $a_1a_2\cdot\cdot\cdot a_d$. It suffices to show that the size of this region decreases each iteration, i.e. that there is a corner position in this region for any barrier.

Start with the tuple $(a_1, a_2, ..., a_d)$ and decrement an $a_i$ in such a way as to yield a new lower position. Repeat this as long as possible. You cannot do this forever, because any lower position satisfies $a_i>0$ for all $i$. Thus, at some point you get stuck. If you're stuck, that means you're at a lower position which only sees upper positions, which is exactly the definition of a corner.
\end{proof}

As alluded to in the caption of figure 8, normal and mis\`ere nim have almost exactly the same $\mathcal{P}$ positions. Barrier isomorphism gives an alternative way to prove this classic result.

\begin{theorem}[Mis\`ere Nim Analysis]
    If every stack in a mis\`ere nim position has size 1, then it is $\mathcal{P}$ if and only if the corresponding normal nim position is $\mathcal{N}$. If the position has any stack with multiple coins, it is $\mathcal{P}$ if and only if the corresponding normal nim position is $\mathcal{P}$.
\end{theorem}

\begin{proof}
    If a every stack in a nim or mis\`ere nim position has a single coin, then there are no meaningful choices to be made. If there is an even number of stacks, the position is $\mathcal{P}$ in normal play and $\mathcal{N}$ in mis\`ere play. If there is an odd number of stacks, then the position is $\mathcal{N}$ in normal play and $\mathcal{P}$ in mis\`ere play. Now we use this, and proceed by induction on $n$ to show that all other $n$-stack mis\`ere nim positions are the same outcome class as their normal counterparts.

    Base case: If $n=1$, then every position with more than 1 coin is $\mathcal{N}$ in both kinds of play, as the first player to move can win by either reducing to 0 coins in the normal case or 1 coin in the mis\`ere case.

    Inductive step: Suppose any position with less than $n$ stacks satisfies the theorem. Let $W$ be the world with $n$ stack nim positions as its interior, and $n-1$ stack nim positions as its exterior. Let $V$ be the corresponding world of mis\`ere nim positions.

    Construct barriers $B$ and $C$ of $W$ and $V$ respectively, for which the upper sets contain the exterior and the $(1, 1, ..., 1)$ position. I claim there is an isomorphism of barriers between $B$ and $C$. This would imply by the barrier isomorphism theorem that all corresponding $n$-stack nim and mis\`ere nim positions with multiple coin stacks are of the same outcome class. The two lower sets have the same image under the lattice map, so it suffices to show that shadowed rays correspond to shadowed rays. There are two cases to consider:
    \begin{enumerate}
        \item Rays which contain $(1, 1, ...., 1, 1)$. These rays also contain some $(1, ..., 1, 0, 1, ..., 1)$. One of these is $\mathcal{P}$ is mis\`ere nim, and the other is $\mathcal{P}$ in normal play, so rays of this type are shadowed in both.
        \item Rays which do not contain $(1, 1, ..., 1, 1)$ intersect the upper set in exactly one position. It is exterior, so it has $n-1$ stacks. By the inductive hypothesis, the position is $\mathcal{P}$ in mis\`ere play if and only if it is $\mathcal{P}$ in normal play, so a ray of this type is shadowed in $W$ if and only if its corresponding ray is shadowed in $V$.
    \end{enumerate}
    This establishes an isomorphism of barriers, as desired.
\end{proof}

Barrier isomorphism has powerful implications, is a rather stringent condition. It can be tweaked to give a weaker but more widely applicable tool, formalized in the next two definitions and theorem.

\begin{defi}[Middle set]
    Let $B$ be a barrier for the world $W$. A \text{middle set} for $B$ is a subset $M_B \subset L_B$ such that for any $p \in M_B$ and any $q$ seen by $p$, $q \in U_B \cup M_B$.
\end{defi}

\begin{defi}[Partial Isomorphism of Barriers]
    Let $B$ be a barrier of $W$ with middle set $M_B$, and let $C$ be a barrier of $V$ with middle set $M_C$. There is a partial isomorphism of barriers $B$ and $C$ if there is a translation $t$ such that
    \begin{enumerate}
        \item There exists some translation $t: \mathbb{Z}^d \rightarrow \mathbb{Z}^d$ such that $l_W(M_B) = t(l_V(M_C))$
        \item For all rays $R \subset W$ such that $R \cap M_B \neq \emptyset$, $R \in S(B) \iff l_V^{*-1}tl_W^*(R) \in S(C)$.
    \end{enumerate}
\end{defi}

\begin{theorem}[Partial Barrier Isomorphism Theorem]
    Let $B$ and $C$ be barriers of worlds $W$ and $V$, that are partially isomorphic with middle sets $M_B$ and $M_C$ and translation $t$. If $p \in M_B$, then $p$ is $\mathcal{P}$ if and only if $l_V^{-1}tl_W(p)$ is $\mathcal{P}$. In short, a partial isomorphism links the outcome classes of positions in the middle sets.
\end{theorem}

\begin{proof}
    The proof proceeds identically to that of the full isomorphism theorem.
\end{proof}

\section{Periodicity}

\begin{defi}[Biray]
    A biray is a subset $B \subset W$ of a $d$-dimensional world $W$ satisfying $$l_W(B) = \{ (m, n, a_3, a_4, ..., a_d) : m, n \geq 0, (n, m)\neq (0, 0) \}$$ for some tuple of positive integers $(a_3, a_4, ..., a_d)$. $B_b$ denotes the ray $R=\{r_n\}$ with $l_W(r_n)=(m, b, a_3, a_4, ..., a_d)$. $B_{b,c}$ denotes the position $l_W^{-1}(b, c, a_3, ..., a_d)$.
\end{defi}

\begin{example}
A ray is more or less a position with one unspecified leaf. A biray is more or less a position with two unspecified leaves.
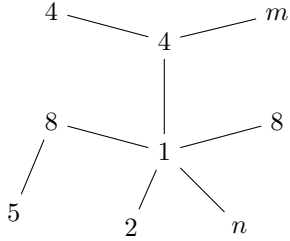
\begin{figure}[!h]
    \centering
    \begin{tikzpicture}
        \node (1) at (0,0) {1};
        \node[above=of 1] (red4) {4};
        \node (red8) at (-1.5, 0.4) {8};
        \node[above=of red8] (green4) {4};
        \node (green8) at (1.5, 0.4) {8};
        \node[above=of green8] (m) {$m$};
        \node (n) at (1, -1) {$n$};
        \node[left=of n] (2) {2};
        \node (5) at (-2, -0.8) {5};

        \draw (green4)--(red4);
        \draw (red4)--(m);
        \draw (red4)--(1);
        \draw (green8)--(1);
        \draw (red8)--(1);
        \draw (red8)--(5);
        \draw (1)--(2);
        \draw (1)--(n);
    \end{tikzpicture}
    \caption{The position $B_{n,m}$ for a certain biray $B$}
    \label{fig:enter-label}
\end{figure}
    
\end{example}

\begin{defi}[to see (as a biray)]
    Let $B$ and $B'$ be birays, not necessarily in the same world. $B$ sees $B'$ (denoted $B \rightarrow B'$) if for all non-negative $(b, c)$ not both 0, $B_{b,c} \rightarrow B'_{b, c}$.
\end{defi}

\begin{defi}[Size of a biray]
    Let $B$ be a biray in a $d$-dimensional world $W$ such that $$l_W(B) = \{ (m, n, a_3, a_4, ..., a_d) : m, n \geq 0, (n, m)\neq (0, 0) \}.$$ Then, the size $|B|$ of $B$ is given by $|B| = a_3+a_4+...+a_d$.
\end{defi}

Notice that if $B \rightarrow B'$ and $B, B' \subset W$, then $|B| > |B'|$.

\begin{defi}[Completion Sequence]
    Given a biray $B$, its completion sequence is given by $\{\mathcal{P}(B_n)\}_{n=1}^\infty$.
\end{defi}

\begin{theorem}[Periodicity Theorem]
    The completion sequence of a biray is additively periodic. Equivalently, for any biray $B$, there exist constants $N_B$ and $\pi_B$ such that for all $n, m > N_B$, $l_W^{-1}(m, n, a_3, ..., a_d)$ is $\mathcal{P}$ if and only if $l_W^{-1}(m+\pi_B, n+\pi_B, a_3, ..., a_d)$ is $\mathcal{P}$.
\end{theorem}

\begin{proof}
    A biray is a subset of a world, which is associated to a tree $T$. Let $\#v$ be the number of vertices of $T$. The proof proceeds by induction on $\#v$.

    Base case: There are no birays for worlds with trees with 0 or 1 vertices. When the world's tree has two vertices, there is one biray, and it is the whole world. On a tree with two vertices, both are leaves, so both are available for play and thus the game is equivalent to normal nim. The completion sequence of the biray is $1, 2, 3, ...$. This is clearly additively periodic, with period $\pi=1$.

    Inductive step: Suppose the theorem holds for all trees with $\#v-1$ vertices. Let $W$ be a $d$-dimensional world whose corresponding tree has $\#v$ vertices. The proof that the theorem holds for all birays $B$ in $W$ proceeds by induction on $|B|$.

    Base case: The base case is more or less the inductive hypothesis of the induction on $\#v$. This will be come clear in a moment.

    Inductive step: Let $B \subset W$ be a biray. Suppose the theorem holds for all birays $B' \subset W$ with $|B'|<|B|$. $B$ sees a number of birays. The birays in $W$ which it sees all have smaller size, and so have additively periodic completion sequences. The rest of the birays $B$ sees live in worlds with smaller corresponding trees. By the first inductive hypothesis, this second class of birays also obeys the periodicity theorem. Let $$B = \{ l_W^{-1}(m, n, a_3, ..., a_d) \;|\;m, n \geq 0, (m, n)\neq (0, 0) \}.$$
    $B$ sees a finite number of birays. To see this, notice that there is a natural one-to-one correspondence $$\{B' \;|\; B \rightarrow B' \} \longleftrightarrow \{ (b_3, b_4, ..., b_d) \;|\; 0 \leq b_i < a_i \text{ for some } i,\; b_j=a_j \text{ for } j\neq i \} $$ and the size of the latter set is equal to the finite quantity $a_3+a_4+...+a_d = |B|$. Let $\pi$ be the least common multiple of the periods of the completion sequences of the birays seen by $B$. Let $N$ be the maximum preperiod length across the same set of sequences. Then, for any biray $B' \leftarrow B$, and $m, n>N$, we have $$B'_{m, n} \text{ is } \mathcal{P} \iff B'_{m+\pi, n+\pi} \text{ is } \mathcal{P}.$$
    Construct a sequence of barriers $B^*_n$,\footnote{I apologize for the syntactic overloading of the letter $B$. In this section, I always include an asterisk when using $B^*$ to denote a barrier. I also use the perhaps strange looking $\pi$ to denote period, to avoid collision with $p$, which represents a position. \textit{$\pi$ is an integer, not the circle constant.}} with $$L_{B^*_n} = [N+n\pi, \infty) \times [N+n\pi, \infty) \times \bigtimes_{i=3}^d [a_i, \infty)$$ and $$M_{B^*_n} = B \cap L_{B^*_n} = [N+n\pi, \infty) \times [N+n\pi, \infty) \times \bigtimes_{i=3}^d \{a_i\}.$$ If any two barriers $B^*_n$ and $B^*_m$ are partially isomorphic with middle sets $M_{B^*_n}$ and $M_{B^*_m}$, then by the partial barrier isomorphism theorem, we will have (letting $m<n$ WLOG) $$\forall k, l > N+mp, \>\>\> B_{k, l} \text{ is } \mathcal{P} \iff B_{k+(n-m)\pi, l+(n-m)\pi} \text{ is } \mathcal{P},$$ Which gives that $B$ has an additively periodic completion sequence with period dividing $(n-m)\pi$.

    Suppose, by way of contradiction, that $B^*_n$ and $B^*_m$ with middle sets $M_{B^*_n}$ and $M_{B^*_m}$ (as defined above) are not partially isomorphic for any $n, m$. It is clear that $M_{B^*_n}$ and $M_{B^*_m}$ are congruent via the translation $$t: \Vec{u} \mapsto \Vec{u} + ((n-m)\pi, (n-m)\pi, 0, 0, ..., 0),$$ so the isomorphism must fail by mapping a shadowed ray to an unshadowed one, or vice versa.

    There are two kinds of rays which intersect the middle set: those which lie in $B$ (variable leaf is one of the first two coordinates) and those which don't (variable leaf is a later coordinate). I claim $t$ and $t^{-1}$ each take shadowed rays of the second type to shadowed rays. If a ray $R$ intersects the middle set $M_{B^*_m}$ but does not lie in $B$, then $$R \cap U_{B^*_m} \subset \bigcup_{B' | B \rightarrow B'} B'.$$ If $R$ is shadowed, then its intersection with the upper set contains some $\mathcal{P}$ position $p=l_W^{-1}(b_0, b_1, ..., b_d) \in B'$. $b_0, b_1 > N$, so $p+(\pi, \pi, 0, 0, ..., 0)$ is $\mathcal{P}$. Applying the periodicity result for $B'$ again, so is $p+(2\pi, 2\pi, 0, 0, ..., 0)$, and inductively so is $p+((n-m)\pi, (n-m)\pi, 0, 0, ..., 0)$. Hence $t(R)$ contains a $\mathcal{P}$ position and thus is also shadowed. The same argument shows $t^{-1}$ takes shadowed type 2 rays to shadowed rays.

    If the would-be isomorphism is going to fail then, it must do so by taking a shadowed type 1 ray to an unshadowed ray, or vice-versa. In effect, the set of shadowed type 1 rays corresponding to each barrier $B^*_n$ partitions the barriers into isomorphism equivalence classes. The assumption we are trying to contradict is that every barrier is in a distinct equivalence class. I claim there are finitely many equivalence classes, which gives by the pigeonhole principle that there exist values $m$ and $n$ for which $B^*_n$ and $B^*_m$ are isomorphic.

    Recall that by the $\mathcal{P}$-completion lemma, for any ray $R$, $\mathcal{P}(R) \leq |R|+1$. Unpacking definitions, this is equivalent to the following triangle inequality-like statement:

    Let $p = l_W^{-1}(a_1, a_2, ..., a_d)$ be a $\mathcal{P}$ position. Then, for each $i$, $$a_i \leq 1 + \sum_{j \neq i} a_j.$$
    Recall that $$B = \{ l_W^{-1}(m, n, a_3, ..., a_d) \;|\; m, n \geq 0, (m, n)\neq (0, 0) \}.$$ 
    The type 1 rays, i.e. those contained in $B$, arise from fixing $m$ or $n$. Let $R_k$ be the ray obtained by fixing $m=k$, and let $R'_l$ be the ray obtained by fixing $n=l$.

    If $R_k$ is shadowed under $B^*_m$, it has an upper $\mathcal{P}$ position $p = l_W^{-1}(k, n, a_3, ..., a_d)$, where $n \leq N+m\pi$. By the $\mathcal{P}$-completion inequality, $$k \leq 1+N+m\pi+\sum_{i=3}^d a_i.$$ Combining with $k \geq N+m\pi$, which is necessary for $R_k$ to intersect $M_{B^*_m}$, there is a finite set of values of $k$ for which $R_k$ may be shadowed. A similar result holds for $l$ and $R'_l$. The isomorphism translations $t_{n, m}$ map the sets of shadowable rays to each other. Thus, each isomorphism equivalence class is characterized by a subset of a finite set. It follows that there are finitely many, completing the proof.
    
\end{proof}





\section{Forest Nim and Grundy Numbers}

The sum game of several tree nim games is called forest nim. The ideas of ray, world, and biray extend easily to this context, and the corresponding theorems still hold, with essentially the same proofs. Notice that we can characterize the Grundy value of a tree nim position as follows: Let $p$ be a tree nim position with corresponding tree $T$. Imagine playing forest nim with underlying forest $T \sqcup v$, where $v$ is an isolated vertex. Then, letting $v$ be the variable leaf and fixing everything on $T$ to match $p$, we get a ray. That ray's $\mathcal{P}$ completion is the Grundy value $\mathcal{G}(p)$ of the original tree nim position. With this understanding, the $\mathcal{P}$ completion lemma and periodicity result give the following corollary regarding tree nim Grundy numbers:

\begin{cor}
    Given a ray $\{r_n\}_{n=0}^\infty$ of tree nim positions, the sequence $\{\mathcal{G}(r_n)\}_{n=0}^\infty$ is additively periodic, and contains every integer once. 
\end{cor}

\begin{remark}
    In \cite{end-nim-1} Albert and Nowakowski note this phenomenon for the specific case of 3-stack end nim. While they do not directly follow from one another, a similar result is known for Wythoff's game \cite{additive-periodicity, fsm-periodicity}, and is proven using similar techniques.
\end{remark}

\section{Tripod Nim}

Tree nim where the tree is a path also goes by the names "End-Nim" and "Burning the Candlestick at Both Ends". This game was solved by Albert and Nowakowski in \cite{end-nim-1}, leaving tripod nim as the simplest unsolved class of tree nim. The full solution and proof are beyond the scope of this paper, but the following result will be helpful later: In 3 stack End-Nim, the $\mathcal{P}$ positions are exactly those of the form $(n, m, n)$ for $n \neq m$. In tripod nim, there are three leaf stacks and one central stack, which is not in play until two leaf stacks are exhausted. The result about 3 stack End-Nim is relevant because tripod nim reduces to 3 stack End-Nim, once the first leaf is exhausted.


\begin{figure}[!b]
    \centering
    \includegraphics[scale=0.4]{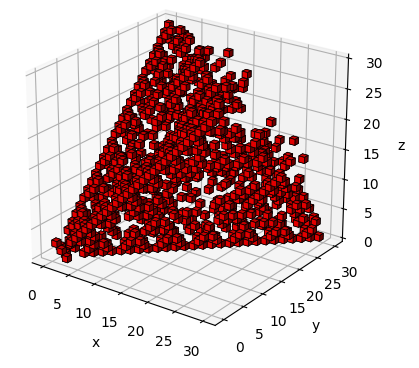}
    \quad
    \includegraphics[scale=0.4]{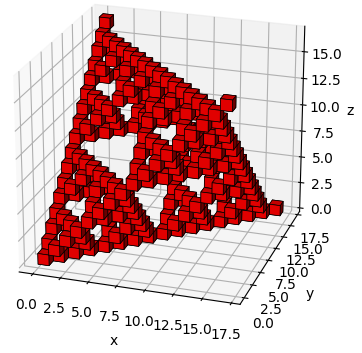}
    \caption{On the left, $\mathcal{P}$ positions for tripod nim with center $=2$. On the right, $\mathcal{P}$ positions for 3 stack nim.}
    \label{tripod_2}
\end{figure}

Fixing some value $o$ for the central vertex gives a 3-dimensional tripod nim world. The interior is the set of tripod nim positions with the specified center, and the exterior is the set of 3-stack end nim positions with $o$ as the central stack. Because this world is 3-dimensional, we can plot the lattice representation, and obtain a plot such as the one below. 

We can also do this for normal nim with 3 stacks, and we get a very pretty Sierpinski pyramid.\cite{serpinski}

\begin{defi}[$C_c(a, b)$]
    The tripod nim positions with center $c$ and sizes $a$ and $b$ for two of the leaves, along with the 3-stack end nim position $(a, c, b)$, form a ray. Define $C_c(a, b)$ to be the $\mathcal{P}$-completion of that ray.
\end{defi}

\begin{remark}
    By the symmetry between the leaves in tripod nim, $C_c(a, b)=C_c(b, a)$. Additionally, $C_k(a, b)=c \iff C_k(b, c)=a \iff C_k(c, a)=b$.
\end{remark}

For a given center $c$, one can construct a table, or infinite matrix, of outputs of $C_c$. This matrix can be thought of as a projection onto $\mathbb{N}^2$ of the structure in figure 3, where $C_c(a, b)$ is the height of the red cell in the column $(a, b)$. Each row and column of the matrix corresponds to a biray. By the $\mathcal{P}$-completion lemma, every row and column of this matrix will contain each non-negative integer exactly once. Furthermore, as a consequence of the symmetry described in remark 1, the matrix is symmetric, and every row and column is an involution. In other words, the corresponding array has the "locator property" explored in \cite{locator-theorem}. To give you a feel for things, table 2 gives a portion of the array for center 6.

\begin{table}[!ht]
\centering
    \begin{tabular}{c|c c c c c c c c c c c c c c c c c c c c c c c}
        &0&1&2&3&4&5&6&7&8&9&10&11&12&13&14&15\\
        \hline
        0&6&1&2&3&4&5&0&7&8&9&10&11&12&13&14&15\\
        1&1&0&3&2&5&4&6&8&7&10&9&12&11&14&13&16\\
        2&2&3&0&1&6&7&4&5&9&8&11&10&13&12&15&14\\
        3&3&2&1&0&7&6&5&4&10&11&8&9&14&15&12&13\\
        4&4&5&6&7&0&1&2&3&11&12&13&8&9&10&16&17\\
        5&5&4&7&6&1&0&3&2&12&13&14&15&8&9&10&11\\
        6&0&6&4&5&2&3&1&9&13&7&12&14&10&8&11&18\\
        7&7&8&5&4&3&2&9&0&1&6&15&13&16&11&17&10\\
    \end{tabular}
    \caption{Tripod Nim With Center $=$ 6}
    \label{tab:table_6}
\end{table}

\begin{remark}
    The given definition of $C_c(a, b)$ only applies to positive $a$ and $b$. This can be extended to pairs where one of $a$ and $b$ is 0 using the symmetry properties outlined in remark 8.1. Additionally, $C_c(0, 0)$ can be taken to equal $c$ by convention. This extends the pattern that every row and column contains each integer exactly once.
\end{remark}

\begin{example}
    The appearance of the value 8 at $(1, 7)$ in table 2 indicates that the position with center 6 and leaves 1, 7, and 8 is a $\mathcal{P}$ position. Denote that tripod nim position as $(6, (1, 7, 8))$, and in general let $(o, (a, b, c))$ refer to the tripod nim position with center $o$ and leaves $a$, $b$, and $c$. 
\end{example}

Understanding the distribution of the values in these matrices is equivalent to understanding tripod nim. Thus, it would nice to be able to generate them without actually referring directly to tripod nim. There are two ways to do this.


Recall that given a world $W$ and barrier $B$, the outcome class of a corner position is determined by whether the position lives in shadowed ray. Applying this inductively, once we have the exterior of the world, the interior is determined. In the 2-D array, the exterior of the world corresponds to the 0th row and column, as well as the 0's in the table. Using the earlier-mentioned analysis of 3 stack End-Nim, we find that for center $c$, the border $\mathcal{P}$ positions in the 3-D lattice will be at coordinates $(0, n, n)$ and their permutations, for all $n \neq c$. This means that upon projection from 3D to 2D, the 0th row and column will each read $c, 1, 2, 3, ..., c-1, 0, c+1, c+2, ...$. This pattern is visible above in table 2. The 0's of the table, aside from the ones at $(0, c)$ and $(c, 0)$, occur along the main diagonal. Every entry of the main diagonal is 0, expect for $(c, c)$.

Now that we have the 0th rows and columns of the table, there are two equivalent ways to fill the rest in. The first is to go each number at a time, first filling in the 0's, then the 1's, then the 2's, and so on. For each number $n$, we fill in row by row, placing $n$ in the earliest available spot which does not have an $n$ directly above it. 

\begin{figure}[!h]
    \centering
    \begin{tikzpicture}
        \matrix at (-0.22, 0.25) [matrix of nodes] 
        {
              & 0 & 1 & 2 & 3 & 4 & 5 & 6 & 7 & 8 & 9 \\
            0 & 6 & 1 & 2 & 3 & 4 & 5 & 0 & 7 & 8 & 9 \\
            1 & 1 & 0 &   & 2 &   &   &   &   &   &   \\
            2 & 2 &   & 0 & 1 &   &   &   &   &   &   \\
            3 & 3 & 2 & 1 & 0 &   &   &   &   &   &   \\
            4 & 4 &   &   &   & 0 & 1 & 2 &   &   &   \\
            5 & 5 & $\times$ & $\times$ & $\times$ & 1 & 0 & $\times$ & |[color=red]| 2 &   &   \\
            6 & 0 &   &   &   &   &   & 1 &   &   &   \\
            7 & 7 &   &   &   &   &   &   & 0 & 1 &   \\
            8 & 8 &   &   &   &   &   &   & 1 & 0 &   \\
            9 & 9 &   &   &   &   &   &   &   &   & 0 \\
        };
        \draw[->] (-1.58, 0.46)--(-1.58, 0);
        \draw[->] (-1.07, 1.86)--(-1.07, 0);
        \draw[->] (-0.56, 1.4)--(-0.56, 0);
        \draw[->] (0.76, 0.1)--(0.76, -0.1);

        \draw (-2.28, 2.75)--(-2.28, -2.35);
        \draw (-2.63, 2.37)--(2.3, 2.37);
        
    \end{tikzpicture}
    \caption{Placing a 2 in row 5 of the center=6 array}
    \label{fig:enter-label}
\end{figure}

\begin{remark}
    Abrams and Cowen-Morton study a family of arrays whose interiors are filled in using the same rules, but whose leading row and column go $n, 0, 1, 2, ...$ instead of $n, 1, 2, ..., n-1, 0, n+1, ...$ \cite{alg-arrays, colorful-arrays, locator-theorem}. These arrays share many of the same properties as those deriving from tripod nim. In particular, the proof of the periodicity theorem for tree nim is a direct generalization of the proof of theorem 4 in \cite{colorful-arrays} to higher dimensions.
\end{remark}

The other way is to fill in row by row. For this method, simply fill each entry with the \textit{minimal excluded value}, or \textbf{mex}, of the entries directly to its left and above it. For example, in table 2, the entry at $(2, 7)$ is 5 because 5 is the \textbf{mex} of the set $\{0, 1, 2, 3, 4, 6, 7, 8\}$, which is the union of $\{2, 3, 0, 1, 6, 7, 4\}$ and $\{7, 8\}$.

\begin{figure}[!htb]
    \centering
    \begin{tikzpicture}
        \matrix at (-0.2, 1.2) [matrix of nodes] 
        {
              & 0 & 1 & 2 & 3 & 4 & 5 & 6 & 7 & 8 & 9 \\
            0 & 6 & 1 & 2 & 3 & 4 & 5 & 0 & 7 & 8 & 9 \\
            1 & 1 & 0 & 3 & 2 & 5 & 4 & 6 & 8 & 7 & 10 \\
            2 & 2 & 3 & 0 & 1 & 6 & 7 & 4 &|[color=red]|5&&\\
            3 & 3 &   &   &   &   &   &   &   &   &   \\
            4 & 4 &   &   &   &   &   &   &   &   &   \\
            5 & 5 &   &   &   &   &   &   &   &   &   \\
        };
        \draw[thin] (-2.1, 1.37) rectangle (0.73, 0.97);
        \draw[thin] (0.78, 2.32) rectangle (1.13, 1.4);

        \draw (-2.63, 2.35)--(2.3, 2.35);
        \draw (-2.15, 2.8)--(-2.15, -0.4);
        
    \end{tikzpicture}
    \caption{Computing $(2, 7)$ in the center$=6$ array}
    \label{fig:enter-label}
\end{figure}

\begin{remark}
    The two above approaches are in fact more or less the same. The difference is one of perspective - the 2-D array is the projection of a 3-D lattice representation onto one of the three axis planes. The two methods correspond to the same process, but viewed with different projection axes.
\end{remark}

\FloatBarrier

\section{Equivalences}

There are certain equivalences between tripod nim worlds with different centers. This can be used to give a full analysis for an infinite family of centers, called the trivial centers.

\subsection{Trivial Centers}

Recall from the proof of the mis\`ere nim analysis that $n$-stack mis\`ere nim is essentially tree nim, where the tree has $n$ leaves protruding from a central vertex of size 1. In particular, tripod nim with center $=1$ is 3-stack mis\`ere nim, and thus has the same $\mathcal{P}$ positions as normal nim, except for when all the leaves are 1.

1 is not the only tripod nim center value for which the $\mathcal{P}$ positions are nearly identical to those of nim. There are similar triviality results for all center values of the form $2^k-1$.

\begin{theorem}[Analysis of Trivial Centers]
    Let $n=2^k-1$ be one less than a power of two. Then, the tripod nim position $(n, (a, b, c))$ is $\mathcal{P}$ if and only if $a \oplus b \oplus c = 0$ or $a=b=c=n$.
\end{theorem}

\begin{proof}

Let $W$ be the world with 3-stack nim as its interior, and let $V$ be the world with center $n$ tripod nim as its interior. Construct barriers $B$ and $C$, whose upper sets $L_B=\partial W$ and $L_C=\partial V$ are the exteriors of their respective worlds. Let $$M_B = l_W^{-1}([1, n-1] \times [1, n-1] \times [1, n])$$ be a middle set for $B$. Define $M_C$ similarly, so that $l_W(M_B)=l_V(M_C)$. Provided all rays which intersect them have the same shadowedness, there is a partial isomorphism relating $M_B$ and $M_C$. 

Let $R\subset W$ be a ray which intersects $M_B$. It intersects the upper set only at some 2-stack nim position $(a, b)$, where $b < n$. This position is $\mathcal{P}$ if and only if $a=b$. The corresponding ray $R'=l^{*-1}_Vl^*_W(R) \subset V$ intersects the upper set at the 3-stack end nim position $(a, n, b)$. Because $b < n$, and in particular $b \neq n$, this position is $\mathcal{P}$ if and only if $a=b$ (using the analysis of 3-stack end nim). Thus, $R$ is shadowed if and only if $R'$ is shadowed, so there is a partial isomorphism equating $M_B$ and $M_C$.

By the partial barrier isomorphism theorem, it follows that the trivial center theorem holds when two leaves are strictly less than the center, and the third is less than or equal to the center. Notice that the above proof did not use the condition that the center is one less than a power of two. Indeed, this part of the result holds for \textit{any} center.

The next step is to show the result holds for all positions with leaves less than or equal to the center. All that's left are positions where two or all three leaves match the center. The set of such positions is the union of 3 rays, which intersect at the position $(n, (n, n, n))$. If this position can be shown to be $\mathcal{P}$, then by the uniqueness of $\mathcal{P}$ completions, it follows that everything else on these rays is $\mathcal{N}$. All the corresponding nim positions are also $\mathcal{N}$. Hence, it suffices to show $(n, (n, n, n))$ is $\mathcal{P}$.

\begin{prop}
    The tripod nim position with all stacks equal is $\mathcal{P}$ exactly when the common value is one less than a power of two.
\end{prop}

\begin{proof}

To see this, notice that the following response, if possible, results in a $\mathcal{P}$ position (by the first part of this proof). 

\FloatBarrier
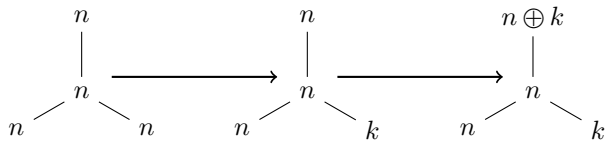
\begin{figure}[!h]
    \centering
    \begin{tikzpicture}
        \node (centerone) at (0, 0) {$n$};
        \node (highleafone) at (0, 1) {$n$};
        \node (leftleafone) at (-0.866, -0.5) {$n$};
        \node (rightleafone) at (0.866, -0.5) {$n$};
        \draw (centerone)--(highleafone);
        \draw (centerone)--(leftleafone);
        \draw (centerone)--(rightleafone);

        \node (centertwo) at (3, 0) {$n$};
        \node (highleaftwo) at (3, 1) {$n$};
        \node (leftleaftwo) at (2.134, -0.5) {$n$};
        \node (rightleaftwo) at (3.866, -0.5) {$k$};
        \draw (centertwo)--(highleaftwo);
        \draw (centertwo)--(leftleaftwo);
        \draw (centertwo)--(rightleaftwo);

        \node (centerthree) at (6, 0) {$n$};
        \node (highleafthree) at (6, 1) {$n \oplus k$};
        \node (leftleafthree) at (5.134, -0.5) {$n$};
        \node (rightleafthree) at (6.866, -0.5) {$k$};
        \draw (centerthree)--(highleafthree);
        \draw (centerthree)--(leftleafthree);
        \draw (centerthree)--(rightleafthree);

        \draw[thick, ->] (0.4, 0.2)--(2.6, 0.2);
        \draw[thick, ->] (3.4, 0.2)--(5.6, 0.2);
    \end{tikzpicture}
    \caption{From a starting position of all $n$'s. The first player reduces a leaf from $n$ to $k$. If possible, the move on another leaf from $n$ to $n \oplus k$ is winning.}
    \label{fig:enter-label}
\end{figure}
\FloatBarrier

For this response to be possible, one needs $n \oplus k < n$. When $n$ is one less than a power of two, its binary expansion is all 1's, so for any lesser non-zero\footnote{No smart player would ever completely remove a leaf, because that would yield the 3-stack end nim position $(n, n, n)$, which is $\mathcal{N}$.} value of $k$, nim-summing $n$ with $k$ will just flip some 1's to 0's. This means that for $n$ one less than a power of two, $(n, (n, n, n))$ is $\mathcal{P}$. If $n$ is not one less than a power of two, then there is always at least one value of $k$ such that $n \oplus k > n$. One example is $k = n \oplus (2^{\lfloor \log_2 n \rfloor + 1} - 1)$. The first player can win by moving to the smallest such value of $k$, so $(n, (n, n, n))$ is an $\mathcal{N}$ position for all $n$ not one less than a power of two.

\end{proof}

With all small positions now covered, a barrier isomorphism can be used to finish the proof. Let $B'$ be a barrier of $W$ with upper set consisting of the exterior, along with interior positions for which every stack has size at most $n$. Let $C'$ be the corresponding barrier of $V$. I claim corresponding rays $R$ and $R'$ which intersect the lower sets have the same shadowedness. There are three cases to consider.

\begin{enumerate}
    \item $R$ intersects $U_{B'}$ at a single position $(a, b)$, with $a>n$ (WLOG). Then, $R'$ intersects $U_{C'}$ at the position $(a, n, b)$, with $a>n$. Each position is $\mathcal{P}$ if and only if $a=b$, so $R$ is shadowed if and only if $R'$ is shadowed.
    \item $R$ contains $(n, n, n)$. Then, $R \cap U_{B'}$ contains the 2-stack nim position $(n, n)$, and thus $R$ is shadowed. $R' \cap U_{C'}$ contains the tripod nim position $(n, (n, n, n))$ and thus $R'$ is also shadowed.
    \item $R \cap U_W = \{ (a, b, n) : n \leq c \}$ for some $a<c$, $b \leq c$. By the partial isomorphism at the beginning of this proof, one of these positions is $\mathcal{P}$ if and only if the corresponding tripod nim position is $\mathcal{P}$, so $R$ is shadowed if and only if $R'$ is shadowed. In fact, rays of this type are always shadowed.
\end{enumerate}

This establishes an isomorphism of barriers $B'$ and $C'$, establishing equivalence for positions where at least one leaf is greater than the center. This completes the proof.
    
\end{proof}

\subsection{The two-five equivalence}

Observe the below plot.

\begin{figure}[!h]
    \centering
    \includegraphics[scale=0.5]{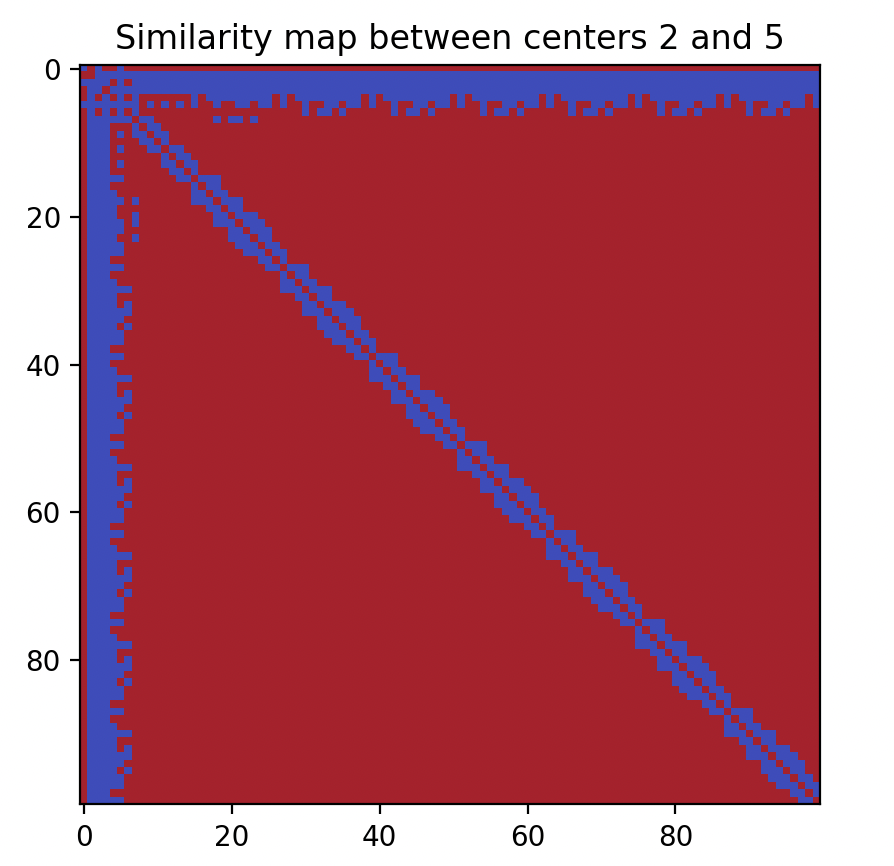}
    \caption{The above plot compares the $\mathcal{P}$-completion arrays for tripod nim with centers 2 and 5. A cell is colored red if its $\mathcal{P}$-completion is the same for both centers, and blue otherwise.}
    \label{fig:enter-label}
\end{figure}

It seems from the plot that aside from the first few rows and columns and a band around the diagonal, that the arrays for tripod nim with 2 and 5 in the center are the same. Indeed, this can be shown by constructing a barrier with lower set $[8, \infty) \times [8, \infty) \times [8, \infty)$ and casting a shadow for each center. Combining the guarantee of the periodicity result with a simple computational check, one can verify that the shadows are indeed identical. Consequently, if the leaves $a$, $b$, and $c$ each take value 8 or greater, then $(2, (a, b, c))$ is $\mathcal{P}$ if and only if $(5, (a, b, c))$ is $\mathcal{P}$.

Call two centers between which there is some barrier isomorphism \textit{nearly equivalent}. Near equivalence is an equivalence relation. Reflexivity and symmetry are trivial. To show transitivity, suppose centers $a$ and $b$ cast the same shadow on a barrier with lower set $L_{ab}$ and $b$ and $c$ cast the same shadow on a barrier with lower set $L_{bc}$. Then, $a$ and $c$ are both the same as $b$ on $L_{ab} \cap L_{bc}$ so they are the same as each other, and so are nearly equivalent, with isomorphic barriers with lower set $L_{ab} \cap L_{bc}$. 

The trivial centers are nearly equivalent to one another, as are 2 and 5. Seeing this, one may hope or conjecture a number of things. Perhaps this relation has finitely many equivalence classes, or perhaps every equivalence class has infinitely many members, or even just multiple. While these statements have not been proven false, there is no computational evidence to support them. Indeed, trying all pairs of integers up to 40, 2-5 seems to be the only near equivalence, other than the trivial centers.

\section{Periods}

Birays in tripod nim correspond to rows of the $\mathcal{P}$-completion array. Specifically, each row is the $\mathcal{P}$-completion sequence of a biray. The periodicity theorem says that the entries in these rows are arithmetico-periodic. What are the periods?

For trivial centers, the period of the $n$th row is the smallest power of two strictly greater than $n$. For non-trivial centers, the periods are much more chaotic. Table 3 lists the periods for the first several rows for small non-trivial centers. Note that these values are found by calculating the first several thousand terms, normally 100,000 to 200,000, and finding the smallest period for which the last half looks arithmetico-periodic. Thus, the larger values have not been rigorously proven correct, although they are strongly suspected to be.

\begin{table}[!p]
    \centering
    \begin{tabular}{c|c c c c c}
         & 2 & 4 & 6 & 8 & 9 \\
         \hline
         0 & 1 & 1 & 1 & 1 & 1 \\
         1 & 2 & 2 & 2 & 2 & 2 \\
         2 & 4 & 2 & 2 & 2 & 2 \\
         3 & 2 & 4 & 4 & 4 & 4 \\
         4 & 12 & 12 & 4 & 4 & 4 \\
         5 & 12 & 6 & 6 & 6 & 6 \\
         6 & 12 & 8 & 6 & 24 & 24 \\
         7 & 8 & 8 & 8 & 24 & 24 \\
         8 & 8 & 12 & 8 & 24 & 48 \\
         9 & 10 & 10 & 10 & 10 & 192 \\
         10 & 60 & 60 & 60 & 60 & 960 \\
         11 & 60 & 60 & 60 & 60 & 3840 \\
         12 & 84 & 84 & 84 & 84 & 26880 \\
         13 & 84 & 84 & 84 & 84 & ? \\
         14 & 84 & 588 & 84 & 84 & ? \\
         15 & 16 & 1764 & 16 & 16 & 16\\
         16 & 16 & 1764 & 16 & 16 & 16 \\
         17 & 18 & 18 & 18 & 18 & 18 \\
         18 & 180 & 180 & 180 & 180 & 180 \\
         19 & 20 & 180 & 180 & 180 & 180 \\
         20 & 900 & 1260 & 1980 & 1980 & 1620\\
         21 & 7200 & 10080 & 1980 & 9990 & ? \\
         22 & 7200 & 10080 & 1980 & 9990 & ? \\
         23 & 24 & 24 & 24 & 24 & 24\\
         24 & 24 & 24 & 312 & 24 & 312 \\
         25 & 26 & 26 & 312 & 26 & 312 \\
         26 & 364 & 364 & 2184 & 364 & 3432 \\
         27 & 364 & 364 & 2184 & 364 & ? \\
         28 & 2184 & 2548 & 6552 & 4732 & ? \\
         29 & ? & 35672 & ? & ? & ? \\
         30 & ? & 35672 & ? & ? & ? \\
         31 & 32 & 32 & 32 & 32 & 32 \\
         32 & 32 & 544 & 544 & 32 & 32 \\
         33 & 34 & 544 & 544 & 34 & 34 \\
         34 & 612 & 544 & 544 & 612 & 612 \\
         35 & 612 & 39168 & 39168 & 612 & 612 \\
         36 & 4896 & ? & ? & 5508 & 6732 \\
         37 & ? & ? & ? & ? & ? \\
         38 & ? & ? & ? & ? & ? \\
         39 & 40 & 40 & 40 & 40 & 40 \\
    \end{tabular}
    \quad
    \begin{tabular}{c|c c c c c c}
         & 2 & 4 & 6 & 8 & 9 \\
         \hline
         40 & 40 & 840 & 840 & 840 & 40 \\
         41 & 42 & 840 & 840 & 840 & 42 \\
         42 & 924 & 8400 & 15120 & 15960 & 924 \\
         43 & 924 & ? & ? & ? & 924 \\
         44 & 21252 & ? & ? & ? & 9240 \\
         45 & ? & ? & ? & ? & ? \\
         46 & ? & ? & ? & ? & ? \\
         47 & 48 & 48 & 48 & 48 & 48 \\
         48 & 48 & 48 & 1200 & 1200 & 48 \\
         49 & 50 & 50 & 1200 & 1200 & 50 \\
         50 & 1300 & 1300 & 1200 & 18000 & 1300 \\
         51 & 1300 & 1300 & ? & ? & 1300 \\
         52 & 32500 & 32500 & ? & ? & 16900 \\
         53 & ? & ? & ? & ? & ? \\
         54 & ? & ? & ? & ? & ? \\
         55 & 56 & 56 & 56 & 56 & 56 \\
         56 & 1624 & 1624 & 1624 & 1624 & 56 \\
         57 & 1624 & 1624 & 1624 & 1624 & 58 \\
         58 & 43848 & 47096 & 47096 & 22736 & 1740 \\
         59 & ? & ? & ? & ? & 1740 \\
         60 & ? & ? & ? & ? & ? \\
         61 & ? & ? & ? & ? & ? \\
         62 & ? & ? & ? & ? & ? \\
         63 & 64 & 64 & 64 & 64 & 64 \\
         64 & 64 & 2112 & 2112 & 2112 & 2112 \\
         65 & 66 & 2112 & 2112 & 2112 & 2112 \\
         66 & 2244 & 35904 & ? & ? & ? \\
         67 & 2244 & ? & ? & ? & ? \\
         68 & ? & ? & ? & ? & ? \\
         69 & ? & ? & ? & ? & ? \\
         70 & ? & ? & ? & ? & ? \\
         71 & ? & 72 & 72 & ? & 27 \\
         72 & ? & 2664 & 2664 & ? & 2664 \\
         73 & 74 & 2664 & 2664 & 74 & 2664 \\
         74 & 2812 & ? & ? & 2812 & ? \\
         75 & 2812 & ? & ? & 2812 & ? \\
         76 & ? & ? & ? & ? & ? \\
         77 & ? & ? & ? & ? & ? \\
         78 & ? & ? & ? & ? & ? \\
         79 & ? & ? & ? & ? & ? \\
    \end{tabular}
    
    \caption{Calculated periods for rows of tripod nim arrays for non-trivial centers between 2 and 9. Center $= 5$ is ommited because it is the same as center $= 2$ except for rows 1 and 2. Question marks indicate unknown values.}
    \label{tab:my_label}
\end{table}

Notice that with rare and small exceptions, whenever the row index is of the form $8n-1$, the corresponding period is $8n$. Additionally, rows $8n$ and $8n+1$ often have periods $8n$ and $8n+2$ respectively. Generally, when they do, row $8n+2$ has period $2(4n+1)(4n+2)$. When they don't, row $8n$ has period $2(4n)(4n+1)$. How should one go about making sense of these patterns?

Recall that by the symmetry between the leaves in tripod nim, the entries $n$th row correspond to the appearance of the value $n$ throughout the array. In particular, $n$ appears at the intersection of column $k$ and row $r$ if and only if $k$ appears as the $r$-index entry in row $n$ (See remark 6.1). Thus, we can investigate the period of the $n$th row by examining the arrangement of the appearances of the number $n$ in the array. These are largely determined by the arrangement of all those numbers less than $n$. So, a natural starting point for investigation is to plot the numbers less than $8n-1$ for various centers and values of $n$ to get a handle on what's going on. Experimentally, it seems that with a few small exceptions, after a certain point the numbers from 0 to $8n-2$ form a perfect band hugging the diagonal. Whenever this happens, it's provably guaranteed that $8n-1$ will have period $8n$, or an even divisor thereof (evenness necessitated by the symmetry of the array). The cleanest proof requires the construction of some machinery, which takes place in the next section.

\begin{figure}[!ht]
    \centering
    \includegraphics[scale=0.5]{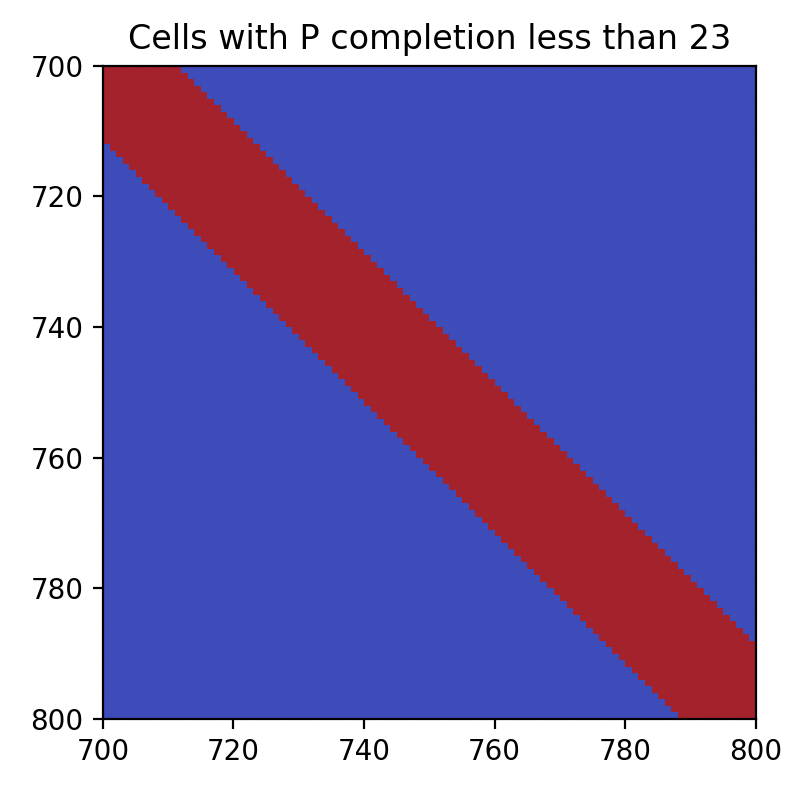}
    \caption{This image comes from the array corresponding to center 6. Cells are colored red if their $\mathcal{P}$ completions are strictly less than 23, and blue otherwise. Notice the axes; this behavior only occurs adequately far from the origin. Near the origin, the plot is much more chaotic.}
    \label{fig:enter-label}
\end{figure}

\FloatBarrier

\begin{remark}
    \textit{Winning Ways}\cite{winning-ways} invites us to consider the game "Third One's Lucky", in which players play nim until there are exactly two coins left, at which point the player to move loses. This game is equivalent to tripod nim with center $=2$, and the book states the $\mathcal{P}$-completion lemma and the periodicity theorem for this special case. It lists the first few periods, which appear in the OEIS as A006018.\cite{third-one-lucky} Analogous sequences for tripod nim with greater non-trivial centers do not return any OEIS results, at the time of writing.
\end{remark}

\section{Distilling to a Dynamical System}

\subsection{Deriving $D(1, n)$}

\begin{example}

Let's explore what the proof of the periodicity theorem looks like for a specific case of tripod nim. Consider the tripod nim world corresponding to center $=2$. The 5th row of the completion array corresponds to a certain biray $B$ - the same one (after a permutation of vertices) as is represented by the set of 5's in the $\mathcal{P}$-completion array. The birays seen by $B$ correspond to the set of 4's, set of 3's, set of 2's, 1's, and 0's. Suppose we've determined each of those satisfies the periodicity theorem. Then, constructing the collection of barriers prescribed in the periodicity proof, we can try to fill in the 5's below each barrier using the procedure outlined in figure 8.

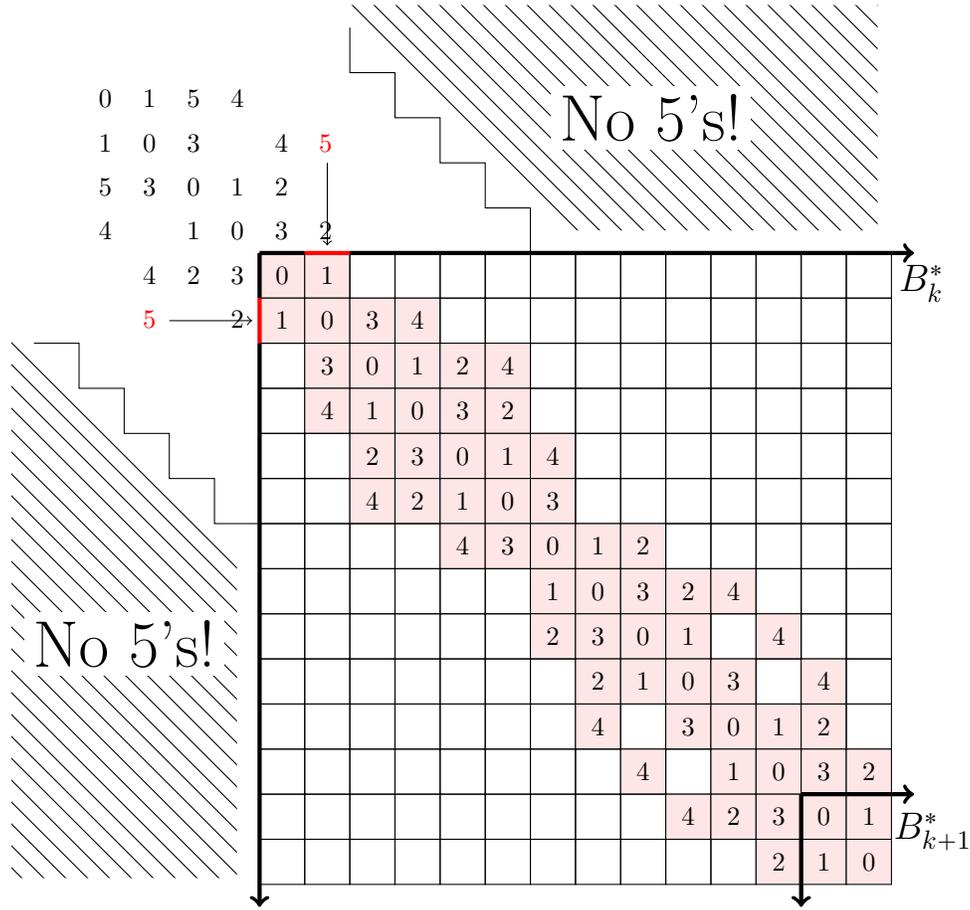
\begin{figure}[!htb]
    \centering
    \begin{tikzpicture}[ne/.style={fill=red!10}]
        \matrix (main) at (0, 0) [matrix of nodes, nodes in empty cells, nodes={rectangle, minimum width=0.6cm, minimum height=0.6cm, anchor=center, draw}, row sep=-\pgflinewidth, column sep=-\pgflinewidth] {
            |[ne]| 0 & |[ne]| 1 & & & & & & & & & & & & \\
            |[ne]| 1 & |[ne]| 0 & |[ne]| 3 & |[ne]| 4 & & & & & & & & & & \\
            & |[ne]| 3 & |[ne]| 0 & |[ne]| 1 & |[ne]| 2 & |[ne]| 4 &&&&&&&&\\
            & |[ne]| 4 & |[ne]| 1 & |[ne]| 0 & |[ne]| 3 & |[ne]| 2 &&&&&&&&\\
            && |[ne]| 2 & |[ne]| 3 & |[ne]| 0 & |[ne]| 1 & |[ne]| 4 &&&&&&&\\
            && |[ne]| 4 & |[ne]| 2 & |[ne]| 1 & |[ne]| 0 & |[ne]| 3 &&&&&&&\\
            &&&& |[ne]| 4 & |[ne]| 3 & |[ne]| 0 & |[ne]| 1 & |[ne]| 2 &&&&&\\
            &&&&&& |[ne]| 1 & |[ne]| 0 & |[ne]| 3 & |[ne]| 2 & |[ne]| 4 &&&\\
            &&&&&& |[ne]| 2 & |[ne]| 3 & |[ne]| 0 & |[ne]| 1 & & |[ne]| 4 &&\\
            &&&&&&& |[ne]| 2 & |[ne]| 1 & |[ne]| 0 & |[ne]| 3 & & |[ne]| 4 &\\
            &&&&&&& |[ne]| 4 & & |[ne]| 3 & |[ne]| 0 & |[ne]| 1 & |[ne]| 2 &\\
            &&&&&&&& |[ne]| 4 & & |[ne]| 1 & |[ne]| 0 & |[ne]|3 & |[ne]| 2 \\
            &&&&&&&&& |[ne]| 4 & |[ne]| 2 & |[ne]| 3 & |[ne]| 0 & |[ne]| 1 \\
            &&&&&&&&&&& |[ne]| 2 & |[ne]| 1 & |[ne]| 0 \\
        };
        
        \draw[ultra thick] (-4.2, 4.2)--(-4.2, 3.6);
        \draw[ultra thick, color=red] (-4.2, 3.6)--(-4.2, 3.0);
        \draw[ultra thick, ->] (-4.2, 3.0)--(-4.2, -4.5);
        \draw[ultra thick] (-4.2, 4.2)--(-3.6, 4.2);
        \draw[ultra thick, color=red] (-3.6, 4.2)--(-3, 4.2);
        \draw[ultra thick, ->] (-3, 4.2)--(4.5, 4.2);
        \node at (4.6, 3.8) {\Large $B^*_k$};

        \draw[ultra thick, ->] (3, -3)--(3, -4.5);
        \draw[ultra thick, ->] (3, -3)--(4.5, -3);
        \node at (4.75, -3.5) {\Large $B^*_{k+1}$};

        \draw (-0.6, 4.2)--(-0.6, 4.8)--(-1.2, 4.8)--(-1.2, 5.4)--(-1.8, 5.4)--(-1.8, 6)--(-2.4, 6)--(-2.4, 6.6)--(-3, 6.6)--(-3, 7.2);
        \draw (-4.2, 0.6)--(-4.8, 0.6)--(-4.8, 1.2)--(-5.4, 1.2)--(-5.4, 1.8)--(-6, 1.8)--(-6, 2.4)--(-6.6, 2.4)--(-6.6, 3)--(-7.2, 3);

        \node
        [
        trapezium, 
        trapezium left angle=125, 
        trapezium right angle=90, 
        trapezium stretches=true,
        minimum height=3cm,
        minimum width=7.2cm,
        pattern=south east lines
        ] at (1.9, 6) {};
        \node[fill=white] at (1, 6) {\Huge No 5's!};

        \node
        [
        trapezium, 
        trapezium left angle=90, 
        trapezium right angle=125, 
        trapezium stretches=true,
        minimum height=3cm,
        minimum width=7.2cm,
        rotate=90,
        pattern=south east lines
        ] at (-6, -2) {};
        \node[fill=white] at (-6, -1) {\Huge No 5's!};

        \matrix (upper) at (-4.2, 4.2) [matrix of nodes, nodes in empty cells, nodes={minimum width=0.6cm, minimum height=0.6cm, anchor=center}, row sep=-\pgflinewidth, column sep=-\pgflinewidth] {
            0 & 1 & 5 & 4 &   &   &   &   \\
            1 & 0 & 3 &   & 4 & |[color=red]| 5 &   &   \\
            5 & 3 & 0 & 1 & 2 &   &   &   \\
            4 &   & 1 & 0 & 3 & 2 &   &   \\
              & 4 & 2 & 3 &   &   &   &   \\
              & |[color=red]| 5 &   & 2 &   &   &   &   \\
              &   &   &   &   &   &   &   \\
              &   &   &   &   &   &   &   \\
        };

        \draw[->] (-5.4, 3.3)--(-4.3, 3.3);
        \draw[->] (-3.3, 5.4)--(-3.3, 4.3);
    \end{tikzpicture}
    \caption{This is a visualization of $B^*_k$ looking down from above, for the specific case of center=2 and $n=5$. The light red cells are shadowed from above, with the numbers in those cells telling which layer has a $\mathcal{P}$ position there. The second row and second column are also shadowed, as indicated by the red segments. Only the first 5 rows or columns could possibly be shadowed, as 5's cannot appear far enough off the diagonal to cast shadows anywhere else. Because the lcm of the periods of 0 to 4 with center 2 is 12, $B^*_{k+1}$ begins 12 rows down from $B^*_k$.}
    \label{fig:enter-label}
\end{figure}

We see that the distribution of the 5's below $B^*_k$ is determined by the distribution of the smaller numbers, as well as by the shadows cast by the earlier 5's. In this case there are only 6 possibly shadowed cells on the upper boundary, and the left boundary matches by symmetry. This leaves 64 possible "seeds" which determine the distribution of 5's below $B^*_k$, guaranteeing an eventual collision.

The distribution of $n$ below each barrier is determined by the corresponding "seed" which is found by reading off the shadow cast on the top edge of the barrier. For example, the seed for $B^*_k$ is 010000, with 0's representing unshadowed segments and 1's representing shadowed ones.

\end{example}

Given the seed for a barrier, and the distribution of the smaller numbers, there is sufficient information to calculate the seed of the next barrier. Thus, the seeds form a dynamical system. 

The key insight relating to the above example is that the necessary condition between the $B^*_k$'s is that the smaller numbers \textit{collectively} fill out the same pattern below each barrier. Taking the lcm of the smaller numbers' periods ensures each individual number fills out the same pattern. This certainly gets the job done, but it is sometimes possible to get away with much smaller gaps between the $B^*_k$'s. In the extreme case, when the numbers less than $n$ form a perfect band hugging the diagonal, we can take successive barriers which are shifts 1 unit down from one another.

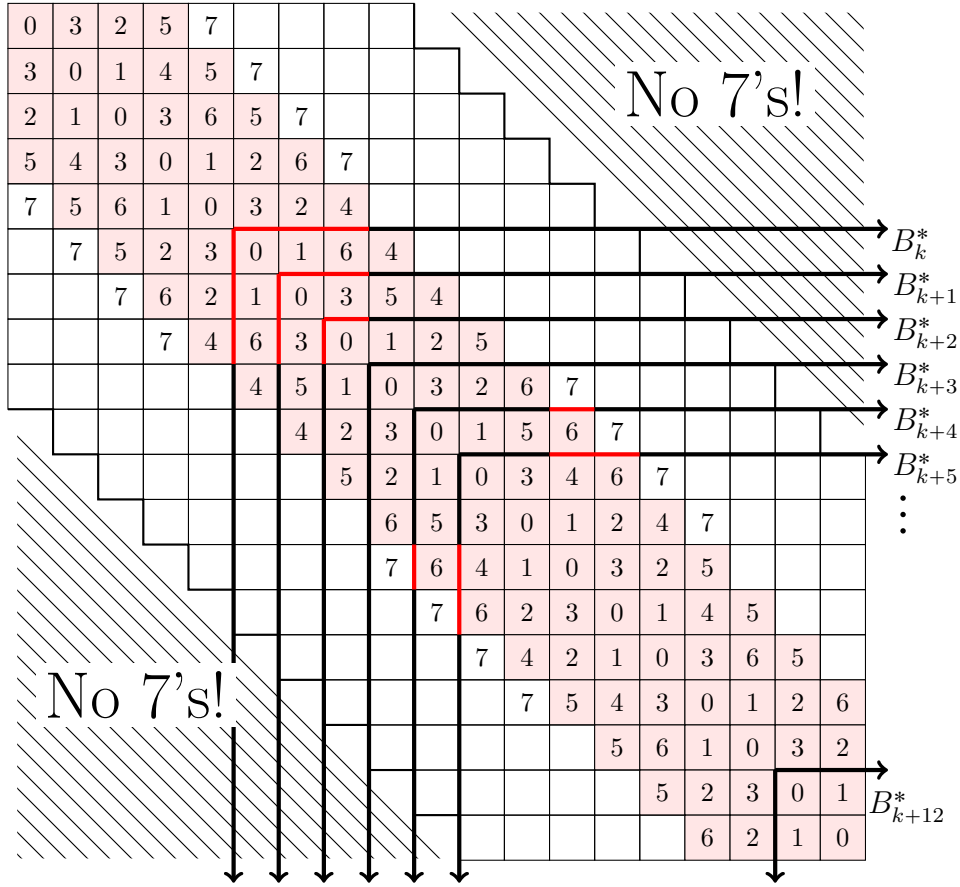
\begin{figure}[!htb]
    \centering
    \begin{tikzpicture}[ne/.style={fill=red!10, draw}, e/.style={draw}]
        \matrix (main) at (-1.5, 1.5) [matrix of nodes, nodes in empty cells, nodes={rectangle, minimum width=0.6cm, minimum height=0.6cm, anchor=center}, row sep=-\pgflinewidth, column sep=-\pgflinewidth] {
            |[ne]| 0 & |[ne]| 3 & |[ne]| 2 & |[ne]| 5 & |[e]|7&|[e]|&|[e]|&|[e]|&|[e]|&&&&&&&&&& \\
            |[ne]| 3 & |[ne]| 0 & |[ne]| 1 & |[ne]| 4 & |[ne]| 5 &|[e]|7&|[e]|&|[e]|&|[e]|&|[e]|&&&&&&&&& \\
            |[ne]| 2 & |[ne]| 1 & |[ne]| 0 & |[ne]| 3 & |[ne]| 6 & |[ne]| 5 &|[e]|7&|[e]|&|[e]|&|[e]|&|[e]|&&&&&&&& \\
            |[ne]| 5 & |[ne]| 4 & |[ne]| 3 & |[ne]| 0 & |[ne]| 1 & |[ne]| 2 & |[ne]| 6 &|[e]|7&|[e]|&|[e]|&|[e]|&|[e]|&&&&&&& \\
            |[e]|7& |[ne]| 5 & |[ne]| 6 & |[ne]| 1 & |[ne]| 0 & |[ne]| 3 & |[ne]| 2 & |[ne]| 4 &|[e]|&|[e]|&|[e]|&|[e]|&|[e]|&&&&&& \\
            |[e]|&|[e]|7& |[ne]| 5 & |[ne]| 2 & |[ne]| 3 & |[ne]| 0 & |[ne]| 1 & |[ne]| 6 & |[ne]| 4 &|[e]|&|[e]|&|[e]|&|[e]|&|[e]|&&&&& \\
            |[e]|&|[e]|&|[e]|7& |[ne]| 6 & |[ne]| 2 & |[ne]| 1 & |[ne]| 0 & |[ne]| 3 & |[ne]| 5 & |[ne]| 4 &|[e]|&|[e]|&|[e]|&|[e]|&|[e]|&&&& \\
            |[e]|&|[e]|&|[e]|&|[e]|7& |[ne]| 4 & |[ne]| 6 & |[ne]| 3 & |[ne]| 0 & |[ne]| 1 & |[ne]| 2 & |[ne]| 5 &|[e]|&|[e]|&|[e]|&|[e]|&|[e]|&&&\\
            |[e]|&|[e]|&|[e]|&|[e]|&|[e]|& |[ne]| 4 & |[ne]| 5 & |[ne]| 1 & |[ne]| 0 & |[ne]| 3 & |[ne]| 2 & |[ne]| 6 &|[e]|7&|[e]|&|[e]|&|[e]|&|[e]|&&\\
            &|[e]|&|[e]|&|[e]|&|[e]|&|[e]|& |[ne]| 4 & |[ne]| 2 & |[ne]| 3 & |[ne]| 0 & |[ne]| 1 & |[ne]| 5 & |[ne]| 6 &|[e]|7&|[e]|&|[e]|&|[e]|&|[e]|&\\
            &&|[e]|&|[e]|&|[e]|&|[e]|&|[e]|& |[ne]| 5 & |[ne]| 2 & |[ne]| 1 & |[ne]| 0 & |[ne]| 3 & |[ne]| 4 & |[ne]| 6 &|[e]|7&|[e]|&|[e]|&|[e]|&|[e]|\\
            &&&|[e]|&|[e]|&|[e]|&|[e]|&|[e]|& |[ne]| 6 & |[ne]| 5 & |[ne]| 3 & |[ne]| 0 & |[ne]| 1 & |[ne]| 2 & |[ne]| 4 &|[e]|7&|[e]|&|[e]|&|[e]|\\
            &&&&|[e]|&|[e]|&|[e]|&|[e]|&|[e]|7& |[ne]| 6 & |[ne]| 4 & |[ne]| 1 & |[ne]| 0 & |[ne]| 3 & |[ne]| 2 & |[ne]| 5 &|[e]|&|[e]|&|[e]|\\
            &&&&&|[e]|&|[e]|&|[e]|&|[e]|&|[e]|7& |[ne]| 6 & |[ne]| 2 & |[ne]| 3 & |[ne]| 0 & |[ne]| 1 & |[ne]| 4 & |[ne]| 5 &|[e]|&|[e]|\\
            &&&&&&|[e]|&|[e]|&|[e]|&|[e]|&|[e]|7& |[ne]| 4 & |[ne]| 2 & |[ne]| 1 & |[ne]| 0 & |[ne]| 3 & |[ne]| 6 & |[ne]| 5 &|[e]|\\
            &&&&&&&|[e]|&|[e]|&|[e]|&|[e]|&|[e]|7& |[ne]| 5 & |[ne]| 4 & |[ne]| 3 & |[ne]| 0 & |[ne]| 1 & |[ne]| 2 & |[ne]| 6 \\
            &&&&&&&&|[e]|&|[e]|&|[e]|&|[e]|&|[e]|& |[ne]| 5 & |[ne]| 6 & |[ne]| 1 & |[ne]| 0 & |[ne]|3 & |[ne]| 2 \\
            &&&&&&&&&|[e]|&|[e]|&|[e]|&|[e]|&|[e]|& |[ne]| 5 & |[ne]| 2 & |[ne]| 3 & |[ne]| 0 & |[ne]| 1 \\
            &&&&&&&&&&|[e]|&|[e]|&|[e]|&|[e]|&|[e]|& |[ne]| 6 & |[ne]| 2 & |[ne]| 1 & |[ne]| 0 \\
        };
        
        \draw[ultra thick, color=red] (-4.2, 4.2)--(-4.2, 2.4);
        \draw[ultra thick, ->] (-4.2, 2.4)--(-4.2, -4.5);
        \draw[ultra thick, color=red] (-4.2, 4.2)--(-2.4, 4.2);
        \draw[ultra thick, ->] (-2.4, 4.2)--(4.5, 4.2);
        \node at (4.8, 4) {\large $B^*_k$};

        \draw[ultra thick, color=red] (-3.6, 3.6)--(-3.6, 2.4);
        \draw[ultra thick, ->] (-3.6, 2.4)--(-3.6, -4.5);
        \draw[ultra thick, color=red] (-3.6, 3.6)--(-2.4, 3.6);
        \draw[ultra thick, ->] (-2.4, 3.6)--(4.5, 3.6);
        \node at (5, 3.4) {\large $B^*_{k+1}$};

        \draw[ultra thick, color=red] (-3, 3)--(-3, 2.4);
        \draw[ultra thick, ->] (-3, 2.4)--(-3, -4.5);
        \draw[ultra thick, color=red] (-3, 3)--(-2.4, 3);
        \draw[ultra thick, ->] (-2.4, 3)--(4.5, 3);
        \node at (5, 2.8) {\large $B^*_{k+2}$};

        \draw[ultra thick, ->] (-2.4, 2.4)--(-2.4, -4.5);
        \draw[ultra thick, ->] (-2.4, 2.4)--(4.5, 2.4);
        \node at (5, 2.2) {\large $B^*_{k+3}$};

        \draw[ultra thick] (-1.8, 1.8)--(-1.8, 0);
        \draw[ultra thick, color=red] (-1.8, 0)--(-1.8, -0.6);
        \draw[ultra thick, ->] (-1.8, -0.6)--(-1.8, -4.5);
        \draw[ultra thick] (-1.8, 1.8)--(0, 1.8);
        \draw[ultra thick, color=red] (0, 1.8)--(0.6, 1.8);
        \draw[ultra thick, ->] (0.6, 1.8)--(4.5, 1.8);
        \node at (5, 1.6) {\large $B^*_{k+4}$};

        \draw[ultra thick] (-1.2, 1.2)--(-1.2, 0);
        \draw[ultra thick, color=red] (-1.2, 0)--(-1.2, -1.2);
        \draw[ultra thick, ->] (-1.2, -1.2)--(-1.2, -4.5);
        \draw[ultra thick] (-1.2, 1.2)--(0, 1.2);
        \draw[ultra thick, color=red] (0, 1.2)--(1.2, 1.2);
        \draw[ultra thick, ->] (1.2, 1.2)--(4.5, 1.2);
        \node at (5, 1) {\large $B^*_{k+5}$};

        \node[rotate=90] at (4.7, 0.4) {\Huge ...};

        \draw[ultra thick, ->] (3, -3)--(3, -4.5);
        \draw[ultra thick, ->] (3, -3)--(4.5, -3);
        \node at (4.75, -3.5) {\large $B^*_{k+12}$};

        \draw[thick] (4.2, 1.2)--(3.6, 1.2)--(3.6, 1.8)--(3, 1.8)--(3, 2.4)--(2.4, 2.4)--(2.4, 3)--(1.8, 3)--(1.8, 3.6)--(1.2, 3.6)--(1.2, 4.2)--(0.6, 4.2)--(0.6, 4.8)--(0, 4.8)--(0, 5.4)--(-0.6, 5.4)--(-0.6, 6)--(-1.2, 6)--(-1.2, 6.6)--(-1.8, 6.6)--(-1.8, 7.2);
        \draw[thick] (-1.2, -4.2)--(-1.2, -3.6)--(-1.8, -3.6)--(-1.8, -3)--(-2.4, -3)--(-2.4, -2.4)--(-3, -2.4)--(-3, -1.8)--(-3.6, -1.8)--(-3.6, -1.2)--(-4.2, -1.2)--(-4.2, -0.6)--(-4.8, -0.6)--(-4.8, 0)--(-5.4, 0)--(-5.4, 0.6)--(-6, 0.6)--(-6, 1.2)--(-6.6, 1.2)--(-6.6, 1.8)--(-7.2, 1.8);

        \node
        [
        isosceles triangle, 
        isosceles triangle apex angle=90, 
        rotate=45,
        minimum height=4cm,
        pattern=south east lines
        ] at (2.5, 5.4) {};
        \node[fill=white] at (2.2, 6) {\Huge No 7's!};

        \node
        [
        isosceles triangle, 
        isosceles triangle apex angle=90, 
        rotate=225,
        minimum height=4cm,
        pattern=south east lines
        ] at (-5.4, -2.5) {};
        \node[fill=white] at (-5.5, -2) {\Huge No 7's!};

    \end{tikzpicture}
    \caption{With center 2, adequately far from the origin, the numbers 0 to 6 form a perfect band hugging the diagonal. Thus, we can take successive barriers for $n=7$ to be shifted one unit down and right of one another. The shadows cast on the barriers are colored bright red.}
    \label{fig:enter-label}
\end{figure}

In such cases, the dynamical system formed by the seeds is relatively simple. Consider the system of seeds corresponding to the number $2n-1$, when all lesser numbers collectively form a perfect band.\footnote{This number is odd because because any perfect band must be made up of an odd number of integers. If a number has an odd number of lesser non-negative integers, it is odd.} To iterate from seed $s_k$ to the next, shift left, dropping leftmost bit, and fill in the rightmost bit with a zero. This accounts for all the shadows cast by values above the barrier $B^*_k$. The only thing left to consider is a $2n-1$ in the row below $B^*_k$ but above $B^*_{k+1}$. If the dropped bit is a 1, then this row has a $2n-1$ to the left of $B^*_k$, so nothing need be added. If the dropped bit is a 0, then there will be a $2n-1$ between the barriers. In accordance with the algorithm featured in figure 8, this $2n-1$ will occupy the first available spot. In summary, the iteration rule is as follows:
\begin{enumerate}
    \item Shift left, filling in with a 0
    \item If the dropped bit is a 0, replace the leftmost 0 which is not in the leftmost $n$ bits with a 1.
\end{enumerate}

\begin{defi}[$D(1, n)$]
    Call the above dynamical system $D(1, n)$.
\end{defi}

\begin{defi}[Radius of a band]
    Given a band which occupies $2n-1$ entries in each row, let $n$ be its radius.
\end{defi}

\subsection{Post-Band Periods}

Recall the claim at the end of section 7. It states that following a band with radius $n$, the next number has even period dividing $2n$. This is implied by the statement that all stable orbits of $D(1, n)$ have period $p\,|\,2n$.

\begin{theorem}[Post-band Period Theorem]
    For any initial state $s \in D(1, n)$, the system is periodic, with orbit dividing $2n$.
\end{theorem}

\begin{defi}[Incubator]
    Elements of $D(1, n)$ (referred to as seeds) are strings of $2n+2$ bits. The leftmost $n$ bits are called the \textit{incubator}. The significance of this portion of the bit string is that, restricted to the incubator, $D(1, n)$'s iteration function is simply a left shift.
\end{defi}

\begin{defi}[Harvest]
    When a seed is iterated, the leftmost bit is dropped as a product of the left-shift. We say that this bit is \textit{harvested}.
\end{defi}

\begin{defi}[energy, simple seed]
    Index the entries in a seed $s \in  D(1, n)$ to the right of the incubator by $1, 2, 3, ...$. Define the \textit{energy} of $s$ to be the greatest index at which a 1 appears. If no 1's appear to the right of the incubator, $s$ has energy 0. Call a seed with 0 energy \textit{simple}.
\end{defi}

\begin{prop}
    Energy goes to 0 after finitely many iterations.
\end{prop}

\begin{proof}

It suffices to show that energy is non-increasing and that non-zero energy must decrease after finitely many successions.

Consider iterating some seed $s$ with energy $g > 0$. If a 0 is harvested, a single 1 is added, which will replace the leftmost 0 among those outside the incubator. There is a 0 at index $g+1$, so this addition happens at some index $i \leq g+1$. Then a leftward shift occurs, decreasing the new 1's maximal index to $g$. Thus, the energy level is at most maintained. If a 1 is harvested, no new 1's are added, and the farthest-right 1 shifts leftward, so energy decreases by 1.

For energy to not decrease, only 0's may be harvested. However, whenever a 0 is harvested, a 1 is added, and that 1 will later be harvested. Thus, non-zero energy cannot forever be maintained. This completes the proof.

\end{proof}

\begin{prop}
    Simple seeds repeat themselves in $2n$ iterations.
\end{prop}

\begin{proof}

Notice that for simple seeds, iteration rules can be simplified to the following:

\begin{itemize}
    \item If the leftmost digit is a 0, replace the $(n+1)$th digit with a 1 and shift left.
    \item If the leftmost digit is a 1, shift left.
\end{itemize}

Essentially, each digit is the opposite of what it was $n$ iterations prior. So, after $2n$ iterations, a simple seed repeats itself.

\end{proof}

The post-band period theorem follows as an immediate corollary of the two propositions.

\subsection{$D(k, n)$ and related conjectures}

The system $D(1, n)$ can be generalized to encode generation seeds for the next $k$ numbers after a band. I.e., if 0 to $2n$ form a band, then the distribution of the numbers $2n+1$ to $2n+k$ are determined by the seeds which act as states for the dynamical system $D(k, n)$, described hereunder. In this generalization, we have $k$ binary sequences stacked in a $k \times (2n+k)$ array of 1's and 0's. As a convention, let binary sequences positioned lower in the array correspond to the positions of smaller entries, and higher-positioned binary sequences correspond to greater entries.

The transition function for $D(k, n)$ can be derived in a similar way to the derivation of $D(1, n)$. The results can be summarized in the following iteration rule:
\begin{enumerate}
    \item Shift left, filling in the rightmost column of the array with 0's.
    \item We are going to now replace a 0 with a 1 in every row from which a 0 was just harvested (think of this as adding the complement of the set we just removed).
    \item We add the 1's one at a time, starting from the bottom and moving upward, according to the following rules: \newline
    a) 1's cannot be added to the leftmost $n$ columns (this corresponds to the band of smaller entries blocking) \newline
    b) Multiple 1's cannot be added to the same column in a single transition \newline
    c) 1's must replace the leftmost possible 0 which adheres to the above two conditions
\end{enumerate}

\begin{example}

Below is a sequence of blocks detailing an example iteration in $D(8, 4)$. The first block is the initial state. The middle shows the block after the shift, and identifies in red which columns need 1's added. The bottom block is the iterated seed, with the added 1's highlighted in red. The dashed line marks the right boundary of the incubator. To the left of the dashed line, no 1's can be added. The incubator contains 4 columns because $n=4$.

\begin{figure}[!ht]
    \centering
    
    \begin{tikzpicture}
        \matrix (start) at (0, 0) [matrix of nodes] 
        {
            1 & 1 & 1 & 1 & 0 & 0 & 0 & 0 & 0 & 0 & 0 \\
            0 & 0 & 0 & 0 & 0 & 0 & 0 & 0 & 0 & 0 & 0 \\
            0 & 0 & 0 & 1 & 1 & 1 & 1 & 0 & 0 & 0 & 0 \\
            1 & 1 & 1 & 1 & 1 & 0 & 0 & 0 & 0 & 0 & 0 \\
            0 & 1 & 0 & 1 & 0 & 1 & 0 & 0 & 0 & 0 & 0 \\
            1 & 0 & 1 & 0 & 0 & 0 & 0 & 0 & 0 & 0 & 0 \\
            0 & 1 & 0 & 0 & 1 & 1 & 0 & 0 & 0 & 0 & 0 \\
            1 & 0 & 0 & 0 & 0 & 0 & 0 & 0 & 0 & 0 & 0 \\
        };

        \matrix (shifted) at (0.4, -4) [matrix of nodes] 
        {
            1 & 1 & 1 & 0 & 0 & 0 & 0 & 0 & 0 & 0 & 0 \\
            0 & 0 & 0 & 0 & 0 & 0 & 0 & 0 & 0 & 0 & 0 \\
            0 & 0 & 1 & 1 & 1 & 1 & 0 & 0 & 0 & 0 & 0 \\
            1 & 1 & 1 & 1 & 0 & 0 & 0 & 0 & 0 & 0 & 0 \\
            1 & 0 & 1 & 0 & 1 & 0 & 0 & 0 & 0 & 0 & 0 \\
            0 & 1 & 0 & 0 & 0 & 0 & 0 & 0 & 0 & 0 & 0 \\
            1 & 0 & 0 & 1 & 1 & 0 & 0 & 0 & 0 & 0 & 0 \\
            0 & 0 & 0 & 0 & 0 & 0 & 0 & 0 & 0 & 0 & 0 \\
        };

        \matrix (bumped) at (-2.6, -4) [matrix of nodes] {
            1 \\
            0 \\
            0 \\
            1 \\
            0 \\
            1 \\
            0 \\
            1 \\
        };

        \matrix (toadd) at (-3.6, -4) [matrix of nodes, color=red] {
            0 \\
            1 \\
            1 \\
            0 \\
            1 \\
            0 \\
            1 \\
            0 \\
        };

        \draw[->] (bumped) -- (toadd);

        \matrix (shifted) at (0.4, -8) [matrix of nodes]
        {
            1 & 1 & 1 & 0 & 0 & 0 & 0 & 0 & 0 & 0 & 0 \\
            0 & 0 & 0 & 0 & |[color=red]| 1 & 0 & 0 & 0 & 0 & 0 & 0 \\
            0 & 0 & 1 & 1 & 1 & 1 & 0 & |[color=red]| 1 & 0 & 0 & 0 \\
            1 & 1 & 1 & 1 & 0 & 0 & 0 & 0 & 0 & 0 & 0 \\
            1 & 0 & 1 & 0 & 1 & 0 & |[color=red]| 1 & 0 & 0 & 0 & 0 \\
            0 & 1 & 0 & 0 & 0 & 0 & 0 & 0 & 0 & 0 & 0 \\
            1 & 0 & 0 & 1 & 1 & |[color=red]| 1 & 0 & 0 & 0 & 0 & 0 \\
            0 & 0 & 0 & 0 & 0 & 0 & 0 & 0 & 0 & 0 & 0 \\
        };
        \draw[dashed, thick] (-0.2, -6.1) -- (-0.2, -10);
        
    \end{tikzpicture}

    \caption{An example iteration of $D(8, 4)$.}
    \label{fig:enter-label}
\end{figure}

\FloatBarrier
    
\end{example}

I conjecture that if the numbers 0 to $8n-2$ form a band, so do the numbers 0 to $8n+6$. Stated formally:

\begin{conjecture}[Band Induction]
    Fix $c \in \mathbb{N}$. If there exists some $N, k$ such that for all $a, b > N$, $|a-b| \leq 4k-1 \iff C_c(a, b) \leq 8k-2$, then there exists an $N'$ such that for all $a, b > N'$, $|a-b| \leq 4k+3 \iff C_c(a, b) \leq 8k+6$.
\end{conjecture}

From a state of a dynamical system $D(k, n)$, one can obtain a sequence by counting the number of 1's in each column. The start seed in figure 14 would have corresponding sequence $4, 4, 3, 4, 3, 3, 1, 0, ...$ and the iterated one has sequence $4, 3, 4, 3, 4, 2, 1, 1, 0, ...$. Call this a state's \textit{derived sequence}.

\begin{defi}(Stability)
    Call a seed $s \in D(2k, n)$ \textit{stable} if its derived sequence goes $k, k, ..., k, k-1, k-2, ..., 2, 1, 0, 0, ...$ with $n$ $k$'s appearing.
\end{defi}

\begin{remark}
    An entry in the derived sequence is the number of appearances of the numbers $a \in [2n, 2n+2k)$ to the left of the band in a particular row. If the numbers from 0 to $2n-1$ form a band and the numbers from $2n+2k-1$ do as well, then the numbers from $2n$ to $2n+2k-1$ must be evenly distributed between the left and right sides of the smaller band. Thus, the band induction conjecture implies that for any natural number $n$, any periodic orbit of $D(8, 4n)$ consists entirely of stable seeds. The converse is in fact also true.
\end{remark}

\begin{prop}
    If every periodic orbit of $D(8, 4n)$ consists entirely of stable seeds, then the band induction conjecture holds.
\end{prop}

\begin{proof}

Suppose a tripod nim array has a band with radius $4n$. Then, the positions of the numbers $8n-1$ to $8n+6$ are described by a trajectory of the system $D(8, 4n)$. We say the integers $8n-1, 8n, ..., 8n+6$ are \textit{accounted for} by $D(8, 4n)$. As stated in remark 9.1, the values in the derived sequence of a seed corresponding to this array count the number of accounted-for numbers to the left of the band in a certain row.

Let $s_k$ and $s_{k+1}$ be consecutive stable seeds corresponding to consecutive barriers $B^*_k$ and $B^*_{k+1}$ in the tripod nim array. To obtain the derived sequence of $s_{k+1}$ from $s_k$, we can first left-shift $s_k$'s derived sequence and then add 1 to each of the entries corresponding to the columns to which we add 1's. Because $s_{k+1}$ is also a stable seed, the transition from $s_k$ to $s_{k+1}$ must involve the addition of 1's to the first 4 columns after the incubator.

This corresponds the presence of four accounted-for numbers in the four spots immediately to the right of the band in the row between the barriers $B^*_k$ and $B^*_{k+1}$. Thus, for a sequence of entirely stable seeds, the 8 numbers accounted for by the seed help form the right half of a perfect band of radius $4n+4$. They also then form a perfect left half by the symmetry of the array.

\end{proof}

The last pattern pointed out in section 7 was that we often see rows $8n$ and $8n+1$ with period $2(4n)(4n+1)$ or rows $8n+2$ and $8n+3$ with period $2(4n+1)(4n+2)$. This behavior would be explained by the following conjecture, if proven.

\begin{conjecture}
    Every periodic orbit of $D(3, n)$ has a period which is a divisor of $2(4n)(4n+1)$.
\end{conjecture}

\begin{remark}
    For small values of $n$, one can calculate all the orbits of $D(3, n)$ and see that the conjecture holds for them. I have confirmed the conjecture for $n\leq 9$.
\end{remark}

\section*{Acknowledgments}

This research was made possible by a Brown University SPRINT/UTRA grant. In addition to the SPRINT/UTRA program, I would like to thank my research mentor Prof. Richard Schwartz for his guidance and many helpful conversations.

\bibliographystyle{plain}
\bibliography{sources.bib}

\begin{thebibliography}{10}

\bibitem{alg-arrays}
Lowell Abrams and Dena~S. Cowen-Morton.
\newblock Algebraic structure in a family of nim-like arrays.
\newblock {\em Journal of Pure and Applied Algebra}, 2010.

\bibitem{colorful-arrays}
Lowell Abrams and Dena~S. Cowen-Morton.
\newblock Periodicity and other structure in a family of nim-like arrays.
\newblock {\em The Electronic Journal of Combinatorics}, 2010.

\bibitem{locator-theorem}
Lowell Abrams and Dena~S. Cowen-Morton.
\newblock A family of nim-like arrays: The locator theorem.
\newblock {\em Theoretical Computer Science}, 2014.

\bibitem{end-nim-1}
M.~H. Albert and R.~J. Nowakowski.
\newblock The game of end nim.
\newblock {\em The Electronic Journal of Combinatorics}, 2001.

\bibitem{winning-ways}
E.R. Berlekamp, J.H. Conway, and R.K. Guy.
\newblock {\em Winning Ways for Your Mathematical Plays, Vol. 3, second edition}.
\newblock AKPeters Ltd., 2001.

\bibitem{first-nim}
C.~L. Bouton.
\newblock Nim, a game with a complete mathematical theory.
\newblock {\em Annals of Mathematics Princeton}, 1902.

\bibitem{end-nim-2}
G.~Cairns and N.~Bao Ho.
\newblock Some remarks on end nim.
\newblock {\em International Journal of Combinatorics}, 2011.

\bibitem{additive-periodicity}
A.~Dress, A.~Flammenkamp, and N.~Pink.
\newblock Additive periodicity of the spraguegrundy function of certain nim games.
\newblock {\em Adv. in Appl. Math.}, 1999.

\bibitem{serpinski}
Kevin Gibbons.
\newblock The geometry of nim.
\newblock {\em arXiv}, 2011.

\bibitem{unsolved-problems}
Richard~K. Guy.
\newblock Unsolved problems in combinatorial games.
\newblock {\em Games of No Chance}, 1998.

\bibitem{fsm-periodicity}
Howard~A. Landman.
\newblock A simple fsm-based proof of the additive periodicity of the sprague-grundy function of wythoff’s game.
\newblock {\em More Games of No Chance}, 2002.

\bibitem{third-one-lucky}
\url{https://oeis.org/A006018}.
\newblock Periods for Third One's Lucky.

\end{thebibliography}

\end{document}